\documentclass[a4paper,12pt]{amsart}
\usepackage{times,fullpage}
\usepackage{latexsym,amscd,amssymb,url}
\usepackage[latin1]{inputenc} 	
\usepackage[T1]{fontenc}        
\usepackage{ae}									
\usepackage[arrow, matrix, curve]{xy}  
\usepackage{graphicx}
\usepackage{MnSymbol}

\newtheorem{theorem}{Theorem}

\theoremstyle{plain}

\newtheorem{claim}{Claim}

\newtheorem{conjecture}[theorem]{Conjecture}
\newtheorem{corollary}[theorem]{Corollary}

\newtheorem{definition}[theorem]{Definition}

\newtheorem{lemma}[theorem]{Lemma}

\newtheorem{proposition}[theorem]{Proposition}
\newtheorem{remark}[theorem]{Remark}

\newcommand{\Z}{\mathbb Z}
\newcommand{\Q}{\mathbb Q}
\newcommand{\C}{\mathbb C}
\newcommand{\N}{\mathbb N}
\newcommand{\R}{\mathbb R}
\newcommand{\HH}{\mathbb H}
\newcommand{\CP}{\mathbb P}

\newcommand{\im}{\operatorname{im}}

\newcommand{\Aut}{\operatorname{Aut}}

\newcommand{\Sym}{\operatorname{Sym}}
\newcommand{\id}{\operatorname{id}}

\newcommand{\Spec}{\operatorname{Spec}}
\newcommand{\Eig}{\operatorname{Eig}}

\begin{document}

\title[Multiplicative sub-Hodge structures of conjugate varieties]{Multiplicative sub-Hodge structures of conjugate varieties}

\author{Stefan Schreieder} 
\address{Max-Planck-Institut für Mathematik, Vivatsgasse 7, 53111 Bonn, Germany 
}
\email{schreied@math.uni-bonn.de}%
\date{December 5, 2013; \copyright{\ Stefan Schreieder 2013}}

\begin{abstract}
For any subfield $K\subseteq \C$, not contained in an imaginary quadratic extension of $\Q$, we construct conjugate varieties whose algebras of $K$-rational $(p,p)$-classes are not isomorphic. 
This compares to the Hodge conjecture which predicts isomorphisms when $K$ is contained in an imaginary quadratic extension of $\Q$; 
additionally, it shows that the complex Hodge structure on the complex cohomology algebra is not invariant under the $\Aut(\C)$-action on varieties.
In our proofs, we find simply connected conjugate varieties whose multilinear intersection forms on $H^2(-,\R)$ are not (weakly) isomorphic. 
Using these, we detect non-homeomorphic conjugate varieties for any fundamental group and in any birational equivalence class of dimension $\geq 10$.
\end{abstract}

\maketitle

\section{Introduction}
For a smooth complex projective variety $X$ and an automorphism $\sigma$ of $\C$, the conjugate variety $X^\sigma$ is defined via the fiber product diagram 
\[
\xymatrix{
X^\sigma \ar[d] \ar[r] &X\ar[d]\\
\Spec(\C) \ar[r]^{\sigma^\ast} &\Spec(\C) .}
\]
To put it another way, $X^\sigma$ is the smooth variety whose defining equations in some projective space are given by applying $\sigma$ to the coefficients of the equations of $X$.
As abstract schemes -- but in general not as schemes over $\Spec(\C)$ -- $X$ and $X^\sigma$ are isomorphic. 
This has several important consequences for the singular cohomology of conjugate varieties.
For instance, pull-back of forms induces a $\sigma$-linear isomorphism between the algebraic de Rham complexes of $X$ and $X^\sigma$.
This induces an isomorphism of complex Hodge structures
\begin{align} \label{eq:sigmaliniso}
H^\ast(X,\C)\otimes _\sigma \C \stackrel{\sim}\longrightarrow H^\ast(X^\sigma,\C) ,
\end{align}
where $\otimes _\sigma \C$ means that the tensor product is taken where $\C$ maps to $\C$ via $\sigma$, see \cite{charles-schnell}.
In particular, Hodge and Betti numbers of conjugate varieties coincide. 

The singular cohomology with $\Q_\ell$-coefficients coincides on smooth complex projective varieties with $\ell$-adic \'etale cohomology.
Since \'etale cohomology does not depend on the structure morphism to $\Spec(\C)$, we obtain isomorphisms of graded $\Q_\ell$-, resp.\ $\C$-algebras,
\begin{align} \label{eq:etaleiso}
H^\ast(X,\Q_\ell) \stackrel{\sim}\longrightarrow H^\ast(X^\sigma,\Q_\ell)\ \ \text{and}\ \ H^\ast(X,\C) \stackrel{\sim}\longrightarrow H^\ast(X^\sigma,\C) ,
\end{align}
depending on an embedding $\Q_\ell\subseteq \C$.
Since the latter isomorphism is $\C$-linear, it is not induced by (\ref{eq:sigmaliniso}).

Only recently, F.\ Charles discovered that there are however aspects of singular cohomology which are not invariant under conjugation:

\begin{theorem}[F.\ Charles \cite{charles}] \label{thm:charles}
There exist conjugate smooth complex projective varieties with distinct real cohomology algebras.
\end{theorem}

It is the aim of this paper to further investigate to which extent cohomological data is invariant under the $\Aut(\C)$-action on varieties.

\subsection{Algebras of $K$-rational $(p,p)$-classes}
For any subfield $K\subseteq \C$, we denote the space of $K$-rational $(p,p)$-classes on $X$ by 
\[
H^{p,p}(X,K):=H^{p,p}(X)\cap H^{2p}(X,K) ;
\]
the corresponding graded $K$-algebra is denoted by $H^{\ast,\ast}(X,K)$.
The Hodge conjecture predicts that $H^{\ast,\ast}(X,\Q)$ is generated by algebraic cycles.
Since each algebraic cycle $Z\subseteq X$ induces a canonical cycle $Z^\sigma\subseteq X^\sigma$ and vice versa, the Hodge conjecture implies

\begin{conjecture}  \label{conj1}
The graded $\Q$-algebra $H^{\ast,\ast}(-, \Q)$ is conjugation invariant.
\end{conjecture}

Apart from the (few) cases where the Hodge conjecture is known, and apart from Deligne's result \cite{deligne} which settles Conjecture \ref{conj1} for abelian varieties, the above conjecture remains wide open, see \cite{charles-schnell,voisin:aspects}.

The above consequence of the Hodge conjecture motivates the investigation of potential conjugation invariance of $H^{\ast,\ast}(-,K)$ for an arbitrary field of coefficients $K\subseteq \C$.
If $K=\Q(iw)$ with $w^2\in \N$ is an imaginary quadratic extension of $\Q$, then the real part, as well as $1/w$ times the imaginary part of a $\Q(iw)$-rational $(p,p)$-class is $\Q$-rational.
Hence,
\[
H^{\ast,\ast}(-,\Q(iw))\cong H^{\ast,\ast}(-,\Q)\otimes _\Q \Q(iw) .
\]
It follows that the Hodge conjecture predicts the conjugation invariance of $H^{\ast,\ast}(-,K)$, when $K$ is contained in an imaginary quadratic extension of $\Q$.
In this paper, we are able to settle all remaining cases:

\begin{theorem} \label{thm:conj1}
Let $K\subseteq \C$ be a subfield, not contained in an imaginary quadratic extension of $\Q$. 
Then there exist conjugate smooth complex projective varieties whose graded algebras of $K$-rational $(p,p)$-classes are not isomorphic.
\end{theorem}

By Theorem \ref{thm:conj1}, there are conjugate smooth complex projective varieties $X$, $X^\sigma$ with
\[
H^{\ast,\ast}(X,\C)\ncong H^{\ast,\ast}(X^\sigma,\C) .
\] 
This shows the following:

\begin{corollary} \label{cor:conj1}
The complex Hodge structure on the complex cohomology algebra of smooth complex projective varieties is not invariant under the $\Aut(\C)$-action on varieties.
\end{corollary}

Corollary \ref{cor:conj1} is in contrast to (\ref{eq:sigmaliniso}) and (\ref{eq:etaleiso}) which show that the complex Hodge structure in each degree, as well as the $\C$-algebra structure of $H^\ast(-,\C)$ are $\Aut(\C)$-invariant. 
The above corollary also shows that there is no embedding $\Q_\ell\hookrightarrow\C$ which guarantees that the isomorphism (\ref{eq:etaleiso}), induced by isomorphisms between $\ell$-adic \'etale cohomologies, respects the complex Hodge structures.

Theorem \ref{thm:conj1} will follow from Theorems \ref{thm:H^{p,p}} and \ref{thm0} below.
Firstly, if $K$ is different from $\R$ and $\C$, then  Theorem \ref{thm:conj1} follows from 

\begin{theorem} \label{thm:H^{p,p}}
Let $K\subseteq \C$ be a subfield, not contained in an imaginary quadratic extension of $\Q$.
If $K$ is different from $\R$ and $\C$, then there exist for any $p\geq 1$ and in any dimension $\geq p+1$ conjugate smooth complex projective varieties $X$, $X^\sigma$ with 
\[
H^{p,p}(X,K) \ncong H^{p,p}(X^\sigma,K) .
\]
\end{theorem}

It is worth noting that Theorem \ref{thm:H^{p,p}} does not remain true if one restricts to smooth complex projective varieties that can be defined over $\overline \Q$, see Remark \ref{rem:Qbar}. 

Next, the case $K=\R$ in Theorem \ref{thm:conj1} follows from the case where $K=\C$ since
\[
H^{\ast,\ast}(X,\R)\otimes_\R \C \cong H^{\ast,\ast}(X,\C) 
\]
holds; so it remains to deal with $K=\C$.
As the isomorphism type of the $\C$-vector space $H^{p,p}(-,\C)$ coincides on conjugate varieties, we now really need to make use of the algebra structure of $H^{\ast,\ast}(-,\C)$. 
Remarkably, it turns out that it suffices to use only a very little amount of the latter, namely the symmetric multilinear intersection form  
\[
H^{1,1}(X,\C)^{\otimes n} \longrightarrow H^{2n}(X,\C) ,
\]
where $n=\dim (X)$.
We explain our result, Theorem \ref{thm0} below, in the next subsection.

\subsection{Multilinear intersection forms on $H^{1,1}(-,K)$ and $H^2(-,K)$}
We say that two symmetric $K$-multilinear forms $V^{\otimes n} \to K$ and $W^{\otimes n} \to K$ on two given $K$-vector spaces $V$ and $W$ are (weakly) isomorphic if there exists a $K$-linear isomorphism $V\cong W$ which respects the given multilinear forms (up to a multiplicative constant).
If $K$ is closed under taking $n$-th roots, then weakly isomorphic intersection forms are already isomorphic.

For a smooth complex projective variety $X$ of dimension $n$, cup product defines symmetric multilinear forms
\[
H^{1,1}(X,K)^{\otimes n}\longrightarrow H^{2n}(X,K)\cong K \ \ \text{and}\ \ H^{2}(X,K)^{\otimes n}\longrightarrow H^{2n}(X,K)\cong K ,
\]
where $H^{2n}(X,K)\cong K$ is the canonical isomorphism that is induced by integrating de Rham classes over $X$.
The weak isomorphism types of the above multilinear forms are determined by the isomorphism types of the graded $K$-algebras $H^{\ast,\ast}(X,K)$ and $H^{2\ast}(X,K)$ respectively.

By the Lefschetz theorem, the Hodge conjecture is true for $(1,1)$-classes and so it is known that the isomorphism type of the intersection form on $H^{1,1}(-,\Q)$ is conjugation invariant.
Additionally, it follows from (\ref{eq:etaleiso}) that the isomorphism types of the intersection forms on $H^{2}(-,\Q_\ell)$ and $H^{2}(-,\C)$ are invariant under conjugation.
Our result, which settles the case $K=\C$ in Theorem \ref{thm:conj1}, contrasts these positive results:

\begin{theorem} \label{thm0}
There exist in any dimension $\geq 4$ simply connected conjugate smooth complex projective varieties whose $\R$-multilinear intersection forms on $H^2(-,\R)$, as well as $\C$-multilinear intersection forms on $H^{1,1}(-,\C)$, are not weakly isomorphic.
\end{theorem}

The examples we will construct in the proof of Theorem \ref{thm0} in Section \ref{sec:thm0'} are defined over cyclotomic number fields.
For instance, one series of examples is defined over $\Q[\zeta_{12}]$; their complex $(1,1)$-classes are spanned by $\Q[\sqrt 3]$-rational ones.
This yields examples $X$, $X^\sigma$ such that the intersection forms on the equidimensional vector spaces $H^{1,1}(X,\Q[\sqrt 3])$ and $H^{1,1}(X^\sigma,\Q[\sqrt 3])$ are not weakly isomorphic, see Corollary \ref{cor:Kg}.

It follows from Theorem \ref{thm0} that the even-degree real cohomology algebra $H^{2\ast}(-,\R)$, as well as the subalgebra $SH^2(-,\R)$ which is generated by $H^2(-,\R)$, is not invariant under conjugation.
Since Charles's examples have dimension $\geq 12$ and fundamental group $\Z^8$, Theorem \ref{thm0} generalizes Theorem \ref{thm:charles} in several different directions.
Another generalization of  Theorem \ref{thm:charles}, namely Theorem \ref{thm1} below, is explained in the following subsection.

\subsection{Applications to conjugate varieties with given fundamental group.}
Conjugate varieties are homeomorphic in the Zariski topology but in general not in the analytic one.
Historically, this was first observed by Serre in \cite{serre}, who constructed conjugate varieties whose fundamental groups are infinite but non-isomorphic.
The first non-homeomorphic conjugate varieties with finite fundamental group were constructed by Abelson \cite{abelson}.
His construction however only works for non-abelian finite groups which satisfy some strong cohomological condition.

Other examples of conjugate varieties which are not homeomorphic (or weaker: not deformation equivalent) are constructed in \cite{catanese,charles,vakil,rajan,shimada}.
Again, the fundamental groups of these examples are of special shapes.
In particular, our conjugate varieties in Theorem \ref{thm0} are the first known non-homeomorphic examples which are simply connected.
This answers a question, posed more than $15$ years ago by D. Reed in \cite{reed}.
Reed's question was our initial motivation to study conjugate varieties and leads us to the more general problem of determining those fundamental groups for which non-homeomorphic conjugate varieties exist.
Since the fundamental group of smooth varieties is a birational invariant, the problem of detecting non-homeomorphic conjugate varieties in a given birational equivalence class refines this problem.
Building upon the examples we will construct in the proof of Theorem \ref{thm0}, we will be able to prove the following:

\begin{theorem} \label{thm1}
Any birational equivalence class of complex projective varieties in dimension $\geq 10$ contains conjugate smooth complex projective varieties whose even-degree real cohomology algebras are non-isomorphic.
\end{theorem}

Theorem \ref{thm1} implies immediately:

\begin{corollary} \label{cor:pi1}
Let $G$ be the fundamental group of a smooth complex projective variety.
Then there exist conjugate smooth complex projective varieties with fundamental group $G$, but non-isomorphic even-degree real cohomology algebras.
\end{corollary}

In Theorem \ref{thm1'} in Section \ref{sec:deformations} we show that the examples in Theorem \ref{thm1} can be chosen to have non-isotrivial deformations.
This is in contrast to the observation that the previously known non-homeomorphic conjugate varieties tend to be rather rigid, cf.\ Remark \ref{rem:deformations}.

\subsection{Constructions and methods of proof.}
Using products of special surfaces with projective space, we will prove Theorem \ref{thm:H^{p,p}} in Section \ref{sec:H^{p,p}}.
The key idea is to construct real curves in the moduli space of abelian surfaces, respectively Kummer K3 surfaces, on which $\dim(H^{1,1}(-,K))$ is constant.
Using elementary facts about modular forms, we then prove that each of our curves contains a transcendental point, i.e.\ a point whose coordinates are algebraically independent over $\Q$.
The action of $\Aut(\C)$ being transitive on the transcendental points of our moduli spaces, Theorem \ref{thm:H^{p,p}} follows as soon as we have seen that our assumptions on $K$ ensure the existence of two real curves as above on which $\dim(H^{p,p}(-,K))$ takes different (constant) values. 

For the proof of Theorem \ref{thm0} in Section \ref{sec:thm0'} we use Charles--Voisin's method \cite{charles,voisin:inv}, see Section \ref{sec:charles}.
We start with simply connected surfaces $Y\subseteq \CP^N$ with special automorphisms, constructed in Section \ref{sec:surfaces}.
Then we blow-up five smooth subvarieties of $Y\times Y\times \CP^N$, e.g.\ the graphs of automorphisms of $Y$.
In order to keep the dimensions low, we then pass to a complete intersection subvariety $T$ of this blow-up.
If $\dim(T)\geq 4$, then the cohomology of $T$ encodes the action of the automorphisms on $H^2(Y,\R)$ and $H^{1,1}(Y,\C)$.
The latter can change under the $\Aut(\C)$-action, which will be the key ingredient in our proofs.

In order to prove Theorem \ref{thm1} in Section \ref{sec:thm2}, we start with a smooth complex projective variety $Z$ of dimension $\geq 10$, representing a given birational equivalence class.
From our previous results, we will be able to pick a four-dimensional variety $T$ and an automorphism $\sigma$ of $\C$ with $Z\cong Z^\sigma$, such that $T$ and $T^\sigma$ have non-isomorphic even-degree real cohomology algebras.
Since $T$ is four-dimensional, we can embed it into the exceptional divisor of the blow-up $\hat{Z}$ of $Z$ in a point and define $W= Bl_T (\hat{Z})$.
Then, $W^\sigma=Bl_{T^\sigma}(\hat Z^\sigma)$ is birational to $Z^\sigma\cong Z$.
Moreover, we will be able to arrange that $b_2(T)$ is larger than $b_{4}( Z)+4$. 
This will allow us to show that any isomorphism between $H^{2\ast}(W,\R)$ and $H^{2\ast}(W^\sigma,\R)$ induces an isomorphism between $H^{2\ast}(T,\R)$ and $H^{2\ast}(T^\sigma,\R)$.
Theorem \ref{thm1} will follow.

\subsection{Conventions.}
Using Serre's GAGA principle \cite{GAGA}, we usually identify a smooth complex projective variety $X$ with its corresponding analytic space, which is a Kähler manifold.
For a codimension $p$ subvariety $V$ in $X$, we denote the corresponding $(p,p)$-class in $H^\ast(X,\Z)$ by $[V]$.

\section{Preliminaries}
\subsection{Cohomology of blow-ups} \label{subsec:blow-ups}
In this subsection we recall important properties about the cohomology of blow-ups, which we will use (tacitly) throughout Sections \ref{sec:charles}, \ref{sec:thm0'} and \ref{sec:thm2}.
Let $Y\subseteq X$ be Kähler manifolds and let $\tilde X=Bl_Y(X)$ be the blow-up of $X$ in $Y$ with exceptional divisor $D\subseteq \tilde X$.
We then obtain a commutative diagram
\[
\xymatrix{
D \ar[d]^{p} \ar[r]^j &\tilde X\ar[d]^\pi\\
Y \ar[r]^{i} &X ,}
\]
where $i$ denotes the inclusion of $Y$ into $X$ and $j$ denotes the inclusion of the exceptional divisor $D$ into $\tilde X$.
Let $r$ denote the codimension of $Y$ in $X$, then we have the following, see \cite[p.\ 180]{voisin1}.

\begin{lemma} \label{lem:blow-up}
There is an isomorphism of integral Hodge structures
\[
H^k(X,\Z)\oplus \left(\bigoplus _{i=0}^{r-2} H^{k-2i-2}(Y,\Z)\right)\stackrel{\sim}\longrightarrow H^k\left(\tilde X,\Z\right) ,
\]
where on $H^{k-2i-2}(Y,\Z)$, the natural Hodge structure is shifted by $(i+1,i+1)$.
On $H^k(X,\Z)$, the above morphism is given by $\pi^\ast$.
On $H^{k-2i-2}(Y,\Z)$ it is given by $j_\ast\circ h^{i}\circ p^\ast$, where $h$ denotes the cup product with $c_1(\mathcal O_D(1))\in H^2(D,\Z)$ and $j_\ast$ is the Gysin morphism of the inclusion $j:D\hookrightarrow  \tilde X$.
\end{lemma}

By the above lemma, each cohomology class of $\tilde X$ is a sum of pullback classes from $X$ and push forward classes from $D$.
The ring structure on $H^\ast (\tilde X,\Z)$ is therefore uncovered by the following lemma.

\begin{lemma} \label{lem:mult}
Let $\alpha,\beta\in H^\ast(D,\Z)$ and $\eta \in H^\ast(X,\Z)$. 
Then, 
\[
\pi^\ast(\eta)\cup j_\ast(\alpha)=j_\ast(p^\ast (i ^\ast \eta)\cup \alpha)
\ \ \ \text{and}\ \ \ j_\ast (\alpha) \cup j_\ast(\beta) =-j_\ast(h\cup\alpha\cup \beta) ,
\]
where $h=c_1(\mathcal O_D(1))\in H^2(D,\Z)$.
\end{lemma}

\begin{proof}
Using $i\circ p=\pi\circ j$, the first assertion follows immediately from the projection formula. 
For the second assertion, one first proves
\begin{equation} \label{eq1:mult}
j_\ast (\alpha) \cup j_\ast(\beta) =j_\ast (1) \cup j_\ast(\alpha\cup \beta) 
\end{equation}
by realizing that the dual statement in homology holds.
Next, note that $j_\ast(1)=c_1(\mathcal O_{\tilde X}(D))$.
Moreover, the restriction of $\mathcal O_{\tilde X}(D)$ to $D$ is isomorphic to $\mathcal O_D(-1)$.
This implies $-h=j^\ast (j_\ast (1))$ and so the projection formula yields:
\[
-j_\ast(h\cup\alpha\cup \beta)=j_\ast(1)\cup j_\ast(\alpha\cup \beta) .
\]
This concludes the proof by (\ref{eq1:mult}).
\end{proof}

\subsection{Eigenvalues of conjugate endomorphisms}
Let $X$ be a smooth complex projective variety with endomorphism $f$ and let $\sigma$ be an automorphism of $\C$.
Via base change, $f$ induces an endomorphism $f^\sigma$ of $X^\sigma$.
If an explicit embedding of $X$ into some projective space $\CP^N$ with homogeneous coordinates $z=[z_0:\ldots :z_N]$ is given, then $f^\sigma$ is determined by
\[
f^\sigma(\sigma(z)))=\sigma(f(z)) 
\] 
for all $z\in X$, where $\sigma$ acts on each homogeneous coordinate simultaneously.
On cohomology, we obtain linear maps
\[
f^\ast:H^{p,q}(X)\longrightarrow H^{p,q}(X)\ \ \text{and}\ \ (f^\sigma)^\ast:H^{p,q}(X^\sigma)\longrightarrow H^{p,q}(X^\sigma) .
\]
These maps commute with the $\sigma$-linear isomorphism 
\[
H^{p,q}(X)\stackrel{\sim}\longrightarrow H^{p,q}(X^\sigma) 
\]
induced by (\ref{eq:sigmaliniso}).
This observation proves:

\begin{lemma} \label{lem:eigenvalues}
The set of eigenvalues of $(f^\sigma)^\ast$ on $H^{p,q}(X^\sigma)$ is given by the $\sigma$-conjugate of the set of eigenvalues of $f^\ast$ on $H^{p,q}(X)$.
\end{lemma}

\subsection{The $j$-invariant of elliptic curves} \label{subsec:j-inv}
Recall that the $j$-invariant of an elliptic curve $E$ with affine Weierstrass equation 
$
y^2=4x^3-g_2x-g_3 
$
equals 
\[
j(E)=\frac{g_2^3}{g_2^3-27g_3^2} .
\]
Two elliptic curves are isomorphic if and only if their $j$-invariant coincide.
From the above formula, we deduce $j(E^\sigma)=\sigma(j(E))$ for all $\sigma \in \Aut(\C)$.
For an element $\tau$ in the upper half plane $\HH$, we use the notation
\begin{align} \label{def:E_tau}
E_\tau:= \C/(\Z+\tau\Z)\ \ \text{and}\ \ j(\tau):=j(E_\tau) .
\end{align}
Then, $j$ induces an isomorphism between any fundamental domain of the action of the modular group $SL_2(\Z)$ on $\HH$ and $\C$. 
Moreover, $j$ is holomorphic on $\HH$ with a cusp of order one at $i\cdot \infty$.

\subsection{Kummer K3 surfaces and theta constants} \label{subsec:K3}
Let $M \in M_{2}(\C)$ be a symmetric matrix whose imaginary part is positive definite.
Then,
\[
A_M :=\C^2/(\Z^2+M \Z^2)
\]
is a principal polarized abelian surface.
The associated Kummer K3 surface $K3(A_M)$ is the quotient of the blow-up of $A_M$ at its 16 2-torsion points by the involution $\cdot(-1)$. 
Equivalently, $K3(A_M)$ is the blow-up of $A_M/(-1)$ at its 16 singular points.

Let $L_M$ be a symmetric line bundle on $A_M$ which induces the principal polarization on $A_M$.
The linear series $|L_M^{\otimes 2}|$ then defines a morphism $A_M\longrightarrow \CP^3$. 
This morphism induces an isomorphism of $A_M/(-1)$ with a degree four hypersurface 
\[
\left\{F_M=0\right\}\subseteq \CP^3 .
\]
The coefficients of $F_M$ are given by homogeneous degree $12$ expressions in the coordinates of Riemann's second order theta constant $\Theta_2 (M)\in \CP^3$, see \cite{geemen} and also \cite[Example 1.1]{ren}.
This constant is defined as
\begin{align} \label{def:thetaconstant}
\Theta_2(M):=[\Theta_2[0,0](M):\Theta_2[1,0](M):\Theta_2[0,1](M):\Theta_2[1,1](M)] .
\end{align}
Here, for $\delta\in \left\{0,1 \right\}^2$, the complex number $\Theta_2[\delta](M)$ denotes the Fourier series
\begin{align} \label{def:thetaseries}
\Theta_2[\delta](M):=\sum_{n\in\Z^2} e^{2\pi i\cdot Q_M(n+\delta/2)} ,
\end{align}
where $Q_M(z)$ is the quadratic form $z^tMz$, associated to $M$.

The above discussion allows us to calculate conjugates of $K3(A_M)$ explicitly.
\begin{lemma}  \label{lem:K3^sigma}
If $\sigma(\Theta_2(M))=\Theta_2(M')$ holds for some automorphism $\sigma\in \Aut(\C)$, then
\begin{align*}
K3(A_{M})^\sigma\cong K3(A_{M'}) .
\end{align*}
\end{lemma}
\begin{proof}
As mentioned above, the coefficients of $F_M$ and $F_{M'}$ are polynomial expressions in the coordinates of $\Theta_2(M)$ and $\Theta_2(M')$ respectively.
The action of $\sigma$ therefore maps the polynomial $F_M$ to $F_{M'}$ and hence $\left\{F_M=0\right\}$ to $\left\{F_{M'}=0\right\}$.
Moreover, this action maps the 16 singular points of $\left\{F_M=0\right\}$ to the 16 singular points of $\left\{F_{M'}=0\right\}$.
The lemma follows from the above description of $K3(A_{M})$ and $K3(A_{M'})$ as smooth models of $\left\{F_M=0\right\}$ and $\left\{F_{M'}=0\right\}$ respectively.
\end{proof}

\begin{remark} \label{rem:whyK3(A)?}
The linear series $|L_M^{\otimes 3}|$ defines an embedding of $A_M$ into $\CP^8$.
It is in principle possible to use this embedding in order to calculate conjugates $A_M^\sigma$ of $A_M$.
In the preceding section we only presented the analogous (easier) calculation for the associated Kummer K3 surface which will suffice for our purposes. 
\end{remark}

\section{Proof of Theorem \ref{thm:H^{p,p}}} \label{sec:H^{p,p}}
In this section, we prove Theorem \ref{thm:H^{p,p}} from the introduction. 
For this purpose, let us fix a subfield $K\subseteq \C$, different from $\R$ and $\C$, which is not contained in an imaginary quadratic extension of $\Q$.
We then need to construct for any $p\geq 1$ and in any dimension $n\geq p+1$ conjugate smooth complex projective varieties $X$, $X^\sigma$ with $H^{p,p}(X,K)\ncong H^{p,p}(X^\sigma,K)$.
After taking products with $\CP^{n-2}$, it clearly suffices to settle the case $p=1$ and $n=2$.

We denote by $K_\R:=K\cap \R$ the real subfield of $K$.
The proof of Theorem \ref{thm:H^{p,p}} for $p=1$ and $n=2$ is now divided into four different cases.
Cases 1 and 2 deal with $K_\R\neq\Q$; in Cases 3 and 4 we settle $K_\R=\Q$.

In Cases 1--3 we will consider for $\tau\in \HH$ the elliptic curve $E_\tau$ with associated $j$-invariant $j(\tau)$ from (\ref{def:E_tau}), and use the following 

\begin{lemma} \label{lem:thm:H^{p,p}}
Let $L\subset \C$ be a subfield.
Then we have for any $a,b \in \R_{>0}$, 
\[
\dim(H^{1,1}(E_{ia}\times E_{ib},L)) =
\begin{cases}
2,\ \text{if $a/b \notin L$ and $a\cdot b \notin L$,} \\
3,\ \text{if $a/b \in L$ and $a\cdot b \notin L$, or  if $a/b \notin L$ and $a\cdot b \in L$,} \\
4,\ \text{if $a/b \in L$ and $a\cdot b \in L$.} \\
\end{cases}
\]
\end{lemma}

\begin{proof}
For $j=1,2$, we denote the holomorphic coordinate on the $j$-th factor of $E_{ia}\times E_{ib}$ by $z_j=x_j+iy_j$.
Then there are basis elements
\[
\alpha_1,\beta_1 \in H^1(E_{ia},\Z)\ \ \text{and}\ \ \alpha_2,\beta_2 \in H^1(E_{ib},\Z) ,
\]
such that
\[
dz_1=\alpha_1+i a\cdot \beta_1 \in H^{1,0}(E_{ia})\ \ \text{and}\ \ dz_2=\alpha_2+i b \cdot \beta_2 \in H^{1,0}(E_{ib}) .
\]
We deduce that the following four $(1,1)$-classes form a basis of $H^{1,1}(E_{ia}\times E_{ib})$:
\[
\alpha_1\cup \beta_1,\ \ \alpha_2\cup \beta_2,\ \ \alpha_1\cup \alpha_2+ab\cdot \beta_1\cup\beta_2\ \ \text{and}\ \  \alpha_1\cup\beta_2 + (a/b)\cdot \alpha_2 \cup \beta_1 . 
\]
The lemma follows.
\end{proof}

\textbf{Case 1:} $K_\R$ is uncountable.

The restriction of the $j$-invariant to $i\cdot \R_{\geq 1}$ is injective.
Since $K_\R$ is uncountable, it follows that there is some $\lambda\geq 1$ in $K_\R$ such that $j(i\lambda)$ is transcendental.

By assumptions, $K_\R$ is different from $\R$.
The additive action of $K_\R$ on $\R$ has therefore more than one orbit and so $\R_{\geq 1}\setminus K_\R$ is uncountable.
As above, it follows that there is some $\mu \in \R_{\geq 1}\setminus K_\R$ such that $j(i\mu)$ is transcendental.
Hence, there is some $\sigma\in \Aut(\C)$ with $\sigma(j(i \lambda))=j(i\mu)$.
Since $j(i)=1$, it follows from the discussion in Section \ref{subsec:j-inv} that
\[
X:= E_{i \lambda}\times E_{i }  \ \ \text{with}\ \ X^\sigma \cong E_{i \mu}\times E_{i } .
\]
Since $\lambda \in K$ and $\mu \notin K$, it follows from Lemma \ref{lem:thm:H^{p,p}} that $H^{1,1}(X,K)$ and  $H^{1,1}(X^\sigma,K)$ are not equidimensional.
This concludes Case 1.
 
\textbf{Case 2:} $K_\R$ is countable and $K_\R\neq \Q$. 

Here we will need the following lemma.

\begin{lemma} \label{lem:j-rel}
Let $\lambda\in \R_{>0}$ be irrational, and let $U\subseteq \R_{>0}$ be an uncountable subset. 
Then there is some $\mu\in U$ such that $j(\mu)$ and $j(\lambda \mu)$ are algebraically independent over $\Q$. 
\end{lemma}

\begin{proof}
For a contradiction, suppose that $j(\mu)$ and $j(\lambda \mu)$ are algebraically dependent over $\Q$ for all $\mu\in U$.
Since the polynomial ring in two variables over $\Q$ is countable, whereas $U$ is uncountable, we may assume that $j(\mu)$ and $j(\lambda \mu)$ satisfy the same polynomial relation for all $\mu\in U$.
Any uncountable subset of $\R$ contains an accumulation point.
Hence, the identity theorem yields a polynomial relation between the holomorphic functions $j(\tau)$ and $j(\lambda \tau)$ in the variable $\tau\in \HH$.
That is,
\[
\sum_{l=0}^nc_l(j(\tau))\cdot j(\lambda \tau)^l=0 ,
\]
where $c_l(j(\tau))$ is a polynomial in $j(\tau)$ which is nontrivial for $l=n$.
We may assume that $n$ is the minimal integer such that a polynomial relation as above exists.
The modular form $j(\tau)$ does not satisfy any nontrivial polynomial relation since it has a cusp of order one at $i\infty$.
Thus, $n\geq 1$.
For $k\in \Z$, we have $j(\tau)=j(\tau+k)$ and so the above identity yields
\[
\sum_{l=0}^nc_l(j(\tau))\cdot \left(j(\lambda \tau)^l-j(\lambda \tau+\lambda k)^l\right)=0 ,
\]
for all $k\in \Z$.
Since $\lambda$ is irrational, $\lambda \tau$ and $\lambda \tau+\lambda k$ do not lie in the same $SL_2(\Z)$ orbit and so $j(\lambda\tau)-j(\lambda \tau+\lambda k) $ is nonzero for all $k\in \Z$.
Thus,
\[
\sum_{l=1}^n c_l(j(\tau))\cdot \sum_{h=0}^{l-1}j(\lambda \tau)^h j(\lambda \tau+\lambda k)^{l-1-h} =0 .
\]
If we now choose a sequence of integers $(k_m)_{m\geq 1}$ such that $\lambda k_m$ tends to zero modulo $\Z$, then the above identity tends to the identity
\[
\sum_{l=1}^n c_l(j(\tau))\cdot l\cdot j(\lambda \tau)^{l-1} =0 .
\]
This contradicts the minimality of $n$.
Lemma \ref{lem:j-rel} follows.
\end{proof}

Since $K_\R$ is countable, it follows that for any $t >0$,
\[
U_t:=\left\{\mu\in \R_{\geq 1}\ |\ t\mu^2\notin K \right\} 
\]
is uncountable.
By assumptions in Case 2, $K_\R$ contains a positive irrational number $\lambda$.
Additionally, we pick a positive irrational number $\lambda' \notin K$.

Then, by Lemma \ref{lem:j-rel}, there are elements $\mu\in U_\lambda$ and $\mu'\in U_{\lambda'}$ such that $j(i\mu)$ and $j(i\lambda\mu)$, as well as $j(i\mu')$ and $j(i\lambda'\mu')$, are algebraically independent over $\Q$.
It follows that for some $\sigma\in \Aut(\C)$, we have
\[
X:=E_{i\lambda\mu} \times E_{i\mu}\ \ \text{with}\ \ X^\sigma\cong E_{i\lambda'\mu'} \times E_{i\mu'} .
\]
Since $\lambda\in K$ and $\lambda\mu^2,\lambda',\lambda'{\mu'}^2\notin K$, it follows from Lemma \ref{lem:thm:H^{p,p}} that $H^{1,1}(X,K)$ and $H^{1,1}(X^\sigma,K)$ are not equidimensional.
This concludes Case 2.

\textbf{Case 3:} $K$ is uncountable and $K_\R=\Q$.

Since $K$ is uncountable, there are elements $\tau,\tau'\in \HH$ with $\tau,\overline{\tau'}\in K$ such that $j(\tau)$ and $j(\tau')$ are algebraically independent over $\Q$.
Also, there are positive numbers $\mu,\mu'\in \R_{>0}$ with $\mu  \mu', \mu/\mu' \notin K_\R=\Q$ such that $j(i\mu)$ and $j(i\mu')$ are algebraically independent over $\Q$.
For some $\sigma \in \Aut(\C)$, we then have
\[
X:=E_\tau\times E_{\tau'} \ \ \text{with}\ \ X^\sigma\cong E_{i\mu}\times E_{i\mu'} .
\] 
Since $\tau,\overline{\tau'}\in K$, the space $H^{1,1}(X,K)$ is at least three-dimensional.
Conversely, $H^{1,1}(X^\sigma,K)$ is two-dimensional by Lemma \ref{lem:thm:H^{p,p}}.
This concludes Case 3.

\textbf{Case 4:} $K$ is countable and $K_\R=\Q$.

This case is slightly more difficult; instead of products of elliptic curves, we will use Kummer K3 surfaces and their theta constants, see Section \ref{subsec:K3}. 
We begin with the definition of certain families of such surfaces. 
For $t=t_1+it_2\in \C$ with $t_1\neq 0$ and $\mu \in \R_{> 0}$, we consider the symmetric matrix
\[
M(\mu,t):=  i\frac{\mu}{2t_1}\cdot \left( \begin{array}{cc}
 2t_1 & 1 \\
1 & |t|^2 
\end{array} \right) .
\] 
For a suitable choice of $t\in \C$, the matrix $-iM(\mu,t)$ is positive definite for all $\mu>0$ and so the abelian surface $A_{M(\mu,t)}$ as well as its associated Kummer K3 surface exist.
For such $t$, we have the following lemma, where $\hat A$ denotes the dual of the abelian surface $A$.

\begin{lemma} \label{lem:K3(A)}
Let $L\subseteq \C$ be a subfield, let $\mu>0$ and let $t=t_1+it_2\in \C$ such that $-i\cdot M(\mu,t)$ is positive definite.
If $t_1$, $|t|^2$ and $\det(M(\mu,t))$ do not lie in $L$, then
\[
\dim(H^{1,1}(K3(\hat{A}_{M(\mu,t)}),L))=
\begin{cases}
17,\ \text{if}\ \ (|t|^2+2t_1\cdot L)\cap L =\emptyset  , \\
18,\ \text{otherwise} .
\end{cases}
\]
\end{lemma}

\begin{proof}
Fix $t\in \C$ and $\mu>0$ such that $-i\cdot M(\mu,t)$ is positive definite and assume that $t_1$, $|t|^2$ and $\det(M(\mu,t))$ do not lie in $L$.
The rational degree two Hodge structure of a Kummer surface $K3(A)$ is the direct sum of $16$ divisor classes with the degree two Hodge structure of $A$.
It therefore remains to investigate the dimension of $H^{1,1}(\hat{A}_{M(\mu,t)},L)$. 

We denote the holomorphic coordinates on $\C^2$ by $z=(z_1,z_2)$, where $z_j=x_j+iy_j$.
The cohomology of $\hat{A}_{M(\mu,t)}$ is given by the homology of $A_{M(\mu,t)}$ and so 
\[
\alpha_1=dx_1,\ \alpha_2=dx_2,\ \alpha_3=\mu/(2t_1)\cdot \left(2t_1 dy_1+dy_2\right),\  \alpha_4=\mu/(2t_1)\cdot \left( dy_1+|t|^2 dy_2\right) 
\]
form a basis of $H^1(\hat{A}_{M(\mu,t)},\Q)$.
Next, $H^{1,1}(\hat{A}_{M(\mu,t)})$ has basis $dz_1\cup d\overline z_1$, $dz_1\cup d\overline z_2$, $dz_2\cup d\overline z_1$ and $dz_2\cup d\overline z_2$.
This basis can be expressed in terms of $\alpha_j\cup \alpha _k$, where $1\leq j<k \leq 4$.
Applying the Gauß algorithm then yields the following new basis of $H^{1,1}(\hat{A}_{M(\mu,t)})$:
\begin{align*}
\Omega_1&:= \alpha_2\cup \alpha_4+\alpha_1\cup \alpha_3, \\
\Omega_2&:=\alpha_1\cup \alpha_4 -|t|^2 \cdot \alpha_1\cup \alpha_3, \\
\Omega_3&:=\alpha_2\cup \alpha_3-2t_1 \cdot \alpha_1\cup \alpha_3 , \\
\Omega_4&:=\alpha_3\cup \alpha_4-\det(M(\mu,t)) \cdot \alpha_1\cup \alpha_2 .
\end{align*}
From this description it follows that if a linear combination $\sum \lambda_i\Omega_i$ is $L$-rational, then all $\lambda_i$ lie in $L$.
Moreover, since $\det(M(\mu,t))\notin L$, the coefficient $\lambda_4$ needs to vanish.

Since $t_1,|t|^2\notin L$, neither $\Omega_2$ nor $\Omega_3$ is $L$-rational.
We conclude that $H^{1,1}(\hat{A}_{M(\mu,t)},L)$ is two-dimensional if
$
|t|^2+2t_1\cdot l_1 = l_2
$ 
has a solution $l_1,l_2\in L$, and it is one-dimensional otherwise.
The lemma follows.
\end{proof}

In the following we will stick to parameters $t$ that are contained in a sufficiently small neighborhood of $1/3+3i$.
For such $t$, the matrix $-i\cdot M(\mu,t)$ is positive definite.  
The reason for the explicit choice of the base point $1/3+3i$ is due to the fact that it slightly simplifies the proof of the subsequent lemma.
In order to state it, we call a point in $\CP^3$ transcendental if its coordinates in some standard affine chart are algebraically independent over $\Q$.
Equivalently, $z\in \CP^3$ is transcendental if and only if $P(z)\neq 0$ for all nontrivial homogeneous polynomials $P$ with rational coefficients. 
That is, the transcendental points of $\CP^3$ are those which lie in the complement of the (countable) union of hypersurfaces which can be defined over $\Q$.
It is important to note that $\Aut(\C)$ acts transitively on this set of points.

\begin{lemma} \label{lem:theta}
There is a neighborhood $V\subseteq \C$ of $1/3+3i$, such that for all $t=t_1+it_2\in V$ with $1$, $t_1$ and $|t|^2$ linearly independent over $\Q$, 
the following holds.
Any uncountable subset $U\subseteq \R_{>0}$ contains a point $\mu\in U$ with:
\begin{enumerate}
	\item The matrix $-i\cdot M(\mu,t)$ is positive definite.
	\item The determinant of $M(\mu,t)$ is not rational.
	\item The theta constant $\Theta_2(M(\mu,t))$ is a transcendental point of $\CP^3$.
\end{enumerate}
\end{lemma}

\begin{proof}
We define the quadratic form
\[
Q(z):=2t_1 z_1^2+2z_1z_2+ |t|^2z_2^2  ,
\] 
where $z=(z_1,z_2)\in \R^2$.  
For $\delta \in \left\{0,1\right\}^2$, the homogeneous coordinate $\Theta_2[\delta](M(\mu,t))$ of the theta constant $\Theta_2(M(\mu,t))$ is then given by
\begin{align} \label{eq:lem:theta}
\Theta_2[\delta](M(\mu,t)) &=\sum_{n\in\Z^2} \exp\left({-\frac{\pi \mu}{t_1} \cdot Q(n+\delta/2)}\right) ,
\end{align}
see  (\ref{def:thetaseries}).
At the point $t=1/3+3i$, we have
\[
Q(z)|_{t=1/3+3i}=\frac{2}{3} \cdot (z_1+3z_2/2)^2+ \frac{137}{18}\cdot z_2^2 .
\]
This shows that there is a neighborhood $V$ of $1/3+3i$ such that $-i\cdot M(\mu,t)$ is positive definite for all $t\in V$ and all $\mu>0$.
For such $t$, the function in (\ref{eq:lem:theta}) is a modular form in the variable $i\cdot \mu \in \HH$, see \cite{freitag}. 

Let us now pick some $t\in V$ with $1$, $t_1$ and $|t|^2$ linearly independent over $\Q$.
Then $-i\cdot M(\mu,t)$ is positive definite and so $\det(M(\mu,t))$ is a nonzero multiple of $\mu^2$.
After possibly removing countably many points of $U$, we may therefore assume 
\[
\det(M(\mu,t))\notin \Q
\]
for all $\mu\in U$.

For a contradiction, we now assume that there is no $\mu\in U$ such that $\Theta_2(M(\mu,t))$ is a transcendental point of $\CP^3$.
Since the polynomial ring in four variables over $\Q$ is countable, we may then assume that there is one homogeneous polynomial $P$ with $P(\Theta_2(M(\mu,t)))=0$ for all $\mu\in U$.
Since $U\subseteq \R_{>0}$ is uncountable, it contains an accumulation point.
Then the identity theorem yields
\begin{align} \label{eq:P(theta)}
P(\Theta_2(M(-i\tau,t)))=0 ,
\end{align}
where the left hand side is considered as holomorphic function in $\tau \in \HH$.

For $\tau \to i\infty$, the modular form $\Theta_2[\delta](M(-i\tau,t))$ from (\ref{eq:lem:theta}) is dominated by the summand where the exponent $Q(n)$ with $n\in \N^2+\delta$ is minimal.
After possibly shrinking $V$, these minima $n_\delta\in \N^2+\delta$ of $Q(n)$ are given as follows:
\[
n_{0,0}=(0,0),\ n_{1,0}=\pm(1/2,0),\ n_{0,1}=\pm(-1,1/2) \ \ \text{and}\ \ n_{1,1}=\pm(-1/2,1/2). 
\]
Noting that $Q(n_{0,0})$ vanishes, we conclude that for $\tau \to i\infty$, the monomial
\[
\Theta_2[0,0](M(-i\tau,t))^h\cdot \Theta_2[1,0](M(-i\tau,t))^j\cdot 
\Theta_2[0,1](M(-i\tau,t))^k\cdot \Theta_2[1,1](M(-i\tau,t))^l 
\]
is dominated by the summand
\[
2\cdot \exp\left(\frac{\pi i\tau}{t_1}\cdot \left(j\cdot Q(n_{1,0})+k\cdot Q(n_{0,1})+l\cdot Q(n_{1,1})\right)\right) .
\]
The left hand side in (\ref{eq:P(theta)}) is then dominated by those summands for which 
\[
j\cdot Q(n_{1,0})+k\cdot Q(n_{0,1})+l\cdot Q(n_{1,1})
\]
is minimal. 
We will therefore arrive at a contradiction as soon as we have seen that this summand is unique.
That is, it suffices to see that $Q(n_{1,0})$, $Q(n_{0,1})$ and $Q(n_{1,1})$ are linearly independent over $\Q$.
In order to see the latter, we calculate 
\[
Q(n_{1,0})=t_1/2,\ Q(n_{0,1})=|t|^2/4+2t_1-1\ \ \text{and}\ \ Q(n_{1,1})= |t|^2/4+t_1/2-1/2 .
\]
The claim is now obvious since $1$, $t_1$ and $|t|^2$ are linearly independent over $\Q$ by assumptions.
This finishes the proof of the lemma.
\end{proof}

We are now able to conclude Case 4.
Let $V$ be the neighborhood of $1/3+3i$ from Lemma \ref{lem:theta}.
Since $K_\R=\Q$ and since $K$ is not contained in any imaginary quadratic extension of $\Q$, we may pick some $t=t_1+it_2\in K\cap V$ which is not quadratic over $\Q$. 
Then $t_1$ is not rational since otherwise $(t-t_1)^2$ would lie in $K_\R=\Q$, which yielded a quadratic relation for $t$ over $\Q$.
It follows that $1$, $t+\overline{t}=2t_1$ and $t\cdot \overline t=|t|^2$ are linearly independent over $\Q$, as otherwise $\overline t$ would lie in $K$ and so $t+\overline{t}=2t_1\in K_\R=\Q$ were rational. 
Hence, the assumptions of Lemma \ref{lem:theta} are satisfied and so there is some $\mu\in \R_{>0}$ such that the pair $(\mu,t)$ satisfies (1)--(3) in Lemma \ref{lem:theta}.

Next, we consider $t'=t_1'+3i\in V$ with $1$, $t_1'$ and $t_1'^2$ linearly independent over $\Q$.
Since $V$ is a neighborhood of $1/3+3i$, there are uncountably many values for $t_1'$ such that $t'$ has the above property.
We claim that we can choose $t_1'$ within this uncountable set such that additionally
\begin{align} \label{eq:t'}
2t'_1 \lambda_1=\lambda_2+|t'|^2
\end{align}
has no solution $\lambda_1,\lambda_2\in K$.
In order to prove this, suppose that $t_1'$ is a solution of (\ref{eq:t'}) for some $\lambda_1,\lambda_2\in K$.
Since $|t'|^2$ is a real number, it follows that $t_1'$ lies in the set of quotients $x/y$ where $x$ and $y$ are imaginary parts of some elements of $K$.
Since $K$ is countable, so is the latter set.
Our claim follows since we can choose $t_1'$ within an uncountable set.
That is, we have just shown that there is a point $t'=t_1'+3i\in V$ with $1$, $t_1'$ and $|t'|^2$ linearly independent over $\Q$ such that additionally, (\ref{eq:t'}) has no solution in $K$.
Then again the assumptions of Lemma \ref{lem:theta} are met and so there is some $\mu'\in \R_{>0}$ such that the pair $(\mu',t')$ satisfies (1)--(3) in Lemma \ref{lem:theta}.

Since $(\mu,t)$ and $(\mu',t')$ satisfy Lemma \ref{lem:theta}, $\Theta_2(M(\mu,t))$ and $\Theta_2(M(\mu',t'))$ are transcendental points of $\CP^3$.
Because $\Aut(\C)$ acts transitively on such points it follows that there is some automorphism $\sigma\in \Aut(\C)$ with
\[
\sigma(\Theta_2(M(\mu,t)))=\Theta_2(M(\mu',t')) .
\]
As the functor $A\mapsto \hat A$ on the category of abelian varieties commutes with the $\Aut(\C)$-action, it therefore follows from Lemma \ref{lem:K3^sigma} that
\[
X:=K3(\hat A_{M(\mu,t)}) \ \ \text{with}\ \ X^\sigma\cong K3(\hat A_{M(\mu',t')}) . 
\] 
By our choices, $t_1$, $|t|$ and $\det(M(\mu,t))$ lie in $\R\setminus \Q$ and the same holds for the pair $(\mu',t')$.
Since $K_\R=\Q$, it follows that $(\mu,t)$ as well as $(\mu',t')$ satisfy the assumptions of Lemma \ref{lem:K3(A)}.
Since (\ref{eq:t'}) has no solution in $K$, whereas 
\[
2t_1\lambda_1=\lambda_2+|t|^2
\] 
has the solution $\lambda_1=t$ and $\lambda_2=t^2$ in $K$, it follows from Lemma \ref{lem:K3(A)} that $H^{1,1}(X,K)$ and $H^{1,1}(X^\sigma,K)$ are not equidimensional.
This concludes Case 4 and hence finishes the proof of Theorem \ref{thm:H^{p,p}}.

\begin{remark} \label{rem:Qbar}
Theorem \ref{thm:H^{p,p}} does not remain true if one restricts to smooth complex projective varieties which can be defined over $\overline \Q$.
Indeed, for each smooth complex projective variety $X$ there is a finitely generated extension $K_X$ of $\Q$ such that for all $p\geq 0$ the group $H^{p,p}(X,\C)$ is generated by $K_X$-rational classes. 
As there are only countably many varieties over $\overline \Q$, it follows that there is an extension $K$ of $\Q$ which is generated by countably many elements such that for each smooth complex projective variety $X$ over $\overline \Q$ and for each $p\geq 0$, the dimension of $H^{p,p}(X,K)$ equals $h^{p,p}(X)$. 
The above claim follows, since $h^{p,p}(X)$ is invariant under conjugation.
\end{remark}

\section{Charles--Voisin's construction} \label{sec:charles}
In this section we carry out a variant of a general construction method due to F. Charles and C. Voisin \cite{charles,voisin:inv}.
The proofs of Propositions \ref{prop:psi} and \ref{prop:psi:(1,1)} below will then be the technical heart of the proof of Theorem \ref{thm0} in Section \ref{sec:thm0'}. 

We start with a smooth complex projective surface $Y$ with $b_1(Y)=0$ and automorphisms $f,f'\in \Aut(Y)$.
Then we pick an embedding
\[
i:Y \hookrightarrow \CP^N
\] 
and assume that $f^\ast$ and $f'^\ast$ fix the pullback $i^\ast h$ of the hyperplane class $h\in H^2(\CP^N,\Z)$.

For a general choice of points $u$, $v$, $w$ and $t$ of $\CP^N$ and $y$ of $Y$, the following smooth subvarieties of $Y\times Y\times \CP^N$ are disjoint:
\begin{align} \label{eq:def:Z_i}
Z_1:=Y\times y\times u,\ \ Z_2:=\Gamma_{\id_Y}\times v,\ \ Z_3:=\Gamma_f\times w,\ \ Z_4:=\Gamma_{f'}\times t,\ \ Z_5:= y\times \Gamma_i ,
\end{align}
where $\Gamma$ denotes the graph of a morphism.
The blow-up 
\[
X:=Bl_{Z_1 \cup \ldots \cup Z_5} \left(Y\times Y \times \CP^N\right)
\]
of $Y\times Y \times \CP^N$ along the union $Z_1\cup\ldots \cup Z_5$ is a smooth complex projective variety. 
Since $b_1(Y)=0$ and $\dim(Y)=2$, it follows from the description of the cohomology of blow-ups, see Section \ref{subsec:blow-ups}, that the cohomology algebra of $X$ is generated by degree two classes.

Next, let $\sigma$ be any automorphism of $\C$.
Then the automorphisms $f$ and $f'$ of $Y$ induce automorphisms $f^{\sigma}$ and $f'^{\sigma}$ of $Y^\sigma$.
Since conjugation commutes with blow-ups, we have
\[
X^{\sigma}=Bl_{Z_1^{\sigma} \cup \ldots \cup Z_5^{\sigma}} \left(Y^{\sigma}\times Y^{\sigma} \times \CP^N\right) ,
\]
where we identified $\CP^N$ with its conjugate ${\CP^N}^\sigma$, and where 
\[
Z_1^\sigma=Y^\sigma\times y^\sigma\times u^\sigma,\ \ Z_2^\sigma=\Gamma_{\id_{Y^\sigma}}\times v^\sigma,\ \ Z_3^\sigma=\Gamma_{f^\sigma}\times w^\sigma,\ \ Z_4^\sigma=\Gamma_{f'^\sigma}\times t^\sigma,\ \ Z_5^\sigma= y^\sigma \times \Gamma_{i^\sigma} .
\]
Here $u^\sigma$, $v^\sigma$, $w^\sigma$ and $t^\sigma$ are points on $\CP^N$, $y^\sigma\in Y^\sigma$, and $i^\sigma:Y^\sigma \hookrightarrow \CP^N$ is the inclusion, induced by $i$. 
The pullback of the hyperplane class via $i^\sigma$ is denoted by ${i^\sigma}^\ast h^\sigma$.

In the next proposition, we will assume that the surface $Y$ has the following properties.
\begin{enumerate}
	\item[(A1)] There exist elements $\alpha,\beta \in H^{1,1}(Y,\Q)$ with $\alpha^2=\beta^2=0$ and $\alpha\cup \beta \neq 0$. 
	\item[(A2)] The sets of eigenvalues of $f^\ast$ and $f'^\ast$ on $H^{2}(Y,\C)$ are distinct.
\end{enumerate}
Then, for a smooth complete intersection subvariety
\[
T\subseteq X , 
\]
with $\dim(T)\geq 4$, the following holds.

\begin{proposition} \label{prop:psi} 
Suppose that (A1) and (A2) hold, and let $K\subseteq \C$ be a subfield.
Then any weak isomorphism between the $K$-multilinear intersection forms on $H^2(T,K)$ and $H^2(T^\sigma,K)$ induces an isomorphism of graded $K$-algebras
\[
\psi: H^\ast(Y,K)\stackrel{\sim}\longrightarrow H^\ast(Y^\sigma,K) ,
\] 
with the following two properties:
\begin{enumerate}
	\item[(P1)] In degree two, $\psi$ maps $i^\ast h$ to a multiple of ${i^\sigma}^\ast h^\sigma$. 
	\item[(P2)] The isomorphism $\psi$ commutes with the induced actions of $f$ and $f'$, i.e.\ 
	\[
\psi \circ f^\ast=(f^\sigma)^\ast \circ  \psi \ \ \text{and}\ \ \psi \circ (f')^\ast=(f^{\prime\sigma})^\ast \circ \psi  .
\] 
\end{enumerate}
\end{proposition}

Proposition \ref{prop:psi} has an analog for isomorphisms between intersection forms on $H^{1,1}(-,K)$. 
In order to state it, we need the following variant of (A2):
\begin{enumerate}
	\item[(A3)] The sets of eigenvalues of $f^\ast$ and $f'^\ast$ on $H^{1,1}(Y,\C)$ are distinct and $\Aut(\C)$-invariant.
\end{enumerate}
Note that $f^\ast$ and $f'^\ast$ are defined on integral cohomology and so their sets of eigenvalues on $H^{2}(Y,\C)$ -- but not on $H^{1,1}(Y,\C)$ -- are automatically $\Aut(\C)$-invariant. 
For this reason, we did not have to impose this additional condition in (A2).

\begin{proposition} \label{prop:psi:(1,1)}
Suppose that (A1) and (A3) hold, and let $K\subseteq \C$ be a subfield which is stable under complex conjugation.
Then any weak isomorphism between the $K$-multilinear intersection forms on $H^{1,1}(T,K)$ and $H^{1,1}(T^\sigma,K)$ induces an isomorphism of graded $K$-algebras 
\[
\psi: H^{\ast,\ast}(Y,K)\stackrel{\sim}\longrightarrow H^{\ast,\ast}(Y^\sigma,K) ,
\] 
which satisfies (P1) and (P2) of Proposition \ref{prop:psi}.
\end{proposition}

\begin{remark}
The assumption (A1) in the above propositions is only needed if $\dim(T)=4$.
\end{remark}

In the following two subsections we prove Propositions \ref{prop:psi} and \ref{prop:psi:(1,1)} respectively;  
important steps will be similar to Charles--Voisin's arguments in \cite{charles,voisin:inv}.

\subsection{Proof of Proposition \ref{prop:psi}} \label{subsec:proof:prop:psi} 
Suppose that there is a $K$-linear isomorphism
\begin{align} \label{eq1:prop:psi}
\phi':H^2(T,K)\stackrel{\sim}\longrightarrow H^2(T^\sigma,K) ,
\end{align}
which induces a weak isomorphism between the respective multilinear intersection forms. 

By the Lefschetz hyperplane theorem, the natural maps
\begin{equation} \label{eq:Lef}
H^k(X,K) \longrightarrow H^k(T,K)\ \ \text{and}\ \ H^k(X^\sigma,K) \longrightarrow H^k(T^\sigma,K)
\end{equation}
are isomorphisms for $k< n$ and injective for $k=n$, where $n:=\dim(T)$.
Using this we will identify classes on $X$ and $X^\sigma$ of degree $\leq n$ with classes on $T$ and $T^\sigma$ respectively.

We denote by $SH^2(-,K)$ the subalgebra of $H^{\ast}(-,K)$ that is generated by $H^2(-,K)$.
Its quotient by all elements of degree $\geq r+1$ is denoted by $SH^2(-,K)^{\leq r}$.
Since $\dim(T)\geq 4$, we obtain from (\ref{eq:Lef}) canonical isomorphisms
\[
SH^2(X,K)^{\leq 4} \stackrel{\sim}{\longrightarrow} SH^2(T,K)^{\leq 4}\ \ \text{and}\ \ SH^{2}(X^\sigma,K)^{\leq 4}\stackrel{\sim}{\longrightarrow}  SH^{2}(T^\sigma,K)^{\leq 4} .
\] 

\begin{claim} \label{claim:phi}
The isomorphism $\phi'$ from (\ref{eq1:prop:psi}) induces a unique isomorphism 
\[
\phi:SH^2(X,K)^{\leq 4} \stackrel{\sim}\longrightarrow SH^{2}(X^\sigma, K)^{\leq 4} 
\] 
of graded $K$-algebras.
\end{claim}

\begin{proof}
In degree two, we define $\phi$ to coincide with $\phi'$ from (\ref{eq1:prop:psi}).
Since the respective algebras are generated in degree two, this determines $\phi$ uniquely as homomorphism of $K$-algebras; we have to check that it is well-defined though.
In order to see the latter, let $\alpha_1,\ldots ,\alpha_r$ and $\beta_1,\ldots, \beta_r$ be elements in $H^2(T,K)$.
Then we have to prove:
\begin{align*} 
\sum_i \alpha_i\cup \beta_i =0 \ \ \Rightarrow \ \ \sum_i \phi'(\alpha_i)\cup \phi'(\beta_i) =0 .
\end{align*}
Let us assume that $\sum_i \alpha_i\cup \beta_i =0$.
Since $\phi'$ induces a weak isomorphism between the corresponding intersection forms, this implies
\[
\sum_i \phi'(\alpha_i)\cup \phi'(\beta_i) \cup \eta=0 \ \ \text{in}\ \ H^{2n}(T^\sigma,K) ,
\]
for all $\eta\in SH^2(T^\sigma,K)^{2n-4}$.
The class $\sum_i \phi'(\alpha_i)\cup \phi'(\beta_i) \cup \eta$ lies in $SH^2(T^\sigma,K)$ and hence it is a pullback of a class on $X$. 
Therefore, the above condition is equivalent to saying that
\[
\sum_i \phi'(\alpha_i)\cup \phi'(\beta_i)\cup \eta \cup [T^\sigma]=0 \ \ \text{in}\ \ H^{2N+8}(X^\sigma,K) ,
\]
for all $\eta\in SH^2(X^\sigma,K)^{2n-4}$. 
Since the cohomology of $X$ is generated by degree two classes, Poincar\'e duality shows
\[
\sum_i \phi'(\alpha_i)\cup \phi'(\beta_i)\cup [T^\sigma]=0 \ \ \text{in}\ \ H^{2N-2n+12}(X^\sigma,K) .
\] 
Since $[T^\sigma]$ is the $(N+4-n)$-th power of some hyperplane class on $X^\sigma$, the Hard Lefschetz theorem implies
\[
\sum_i \phi'(\alpha_i)\cup \phi'(\beta_i)=0 \ \ \text{in}\ \ H^{4}(X^\sigma,K) ,
\]
as we wanted.
Similarly, one proves that $\phi'^{-1}$ induces a well-defined inverse of $\phi$. 
This finishes the proof of the claim.
\end{proof}

From now on, we will work with the isomorphism $\phi$ of $K$-algebras from Claim \ref{claim:phi} instead of the weak isomorphism of intersection forms $\phi'$ from (\ref{eq1:prop:psi}).

In order to describe the degree two cohomology of $X$, we denote by $D_i\subseteq X$ the exceptional divisor above $Z_i$ and we denote by $h$ the pullback of the hyperplane class of $\CP^N$ to $X$.
Then, by Lemma \ref{lem:blow-up}:
\begin{equation} \label{eq:thm0:H^2X}
H^2(X,K)= \left( \bigoplus_{i=1}^5 [D_i]\cdot K \right) \oplus H^2(Y\times Y,K)\oplus h\cdot K . 
\end{equation}
Similarly, we denote by $D_i^\sigma\subseteq X^\sigma$ the conjugate of $D_i$ by $\sigma$ and we denote by $h^\sigma$ the pullback of the hyperplane class of $\CP^N$ to $X^\sigma$. 
This yields:
\begin{equation} \label{eq:thm0:H^2Xsigma}
H^2(X^\sigma,K)= \left( \bigoplus_{i=1}^5 [D_i^\sigma]\cdot K \right) \oplus H^2(Y^\sigma\times Y^\sigma,K)\oplus h^\sigma\cdot K .  
\end{equation}
Next, we pick a base point $0\in Y$ and consider the projections 
\[
Y\times Y\longrightarrow  Y\times 0\ \ \text{and}\ \ Y\times Y\longrightarrow  0\times Y .
\]
Using pullbacks, this allows us to view $H^\ast(Y\times 0,K)$ and $H^\ast(0\times Y,K)$ as subspaces of $H^\ast(Y\times Y,K)$.
By assumption, the first Betti number of $Y$ vanishes and so we have a canonical identity 
\begin{align} \label{eq:H^2(YxY)}
H^2(Y\times Y,K)=H^2(Y\times 0,K) \oplus H^2(0\times Y,K) ,
\end{align}
of subspaces of $H^2(X,K)$.
A similar statement holds on $X^\sigma$.

\begin{claim} \label{claim:H2}
The isomorphism $\phi$ respects the decompositions in (\ref{eq:thm0:H^2X}) and (\ref{eq:thm0:H^2Xsigma}), that is:
\begin{align}
	\phi(H^2(Y\times Y,K))&=H^2(Y^\sigma\times Y^\sigma,K) , \label{eq1:claim:H2} \\
	\phi([D_i]\cdot K)&=[D_i^\sigma]\cdot K \ \ \text{for all $i=1,\ldots ,5$} ,\label{eq2:claim:H2} \\
	\phi(h\cdot K)&=h^\sigma \cdot K . \label{eq3:claim:H2}
\end{align}
\end{claim}
\begin{proof}
In order to prove (\ref{eq1:claim:H2}), we define $S$ to be the linear subspace of $H^2(X,K)$ which is spanned by all classes whose square is zero.
By the ring structure of the cohomology of blow-ups (cf.\ Lemma \ref{lem:mult}), $S$ is contained in $H^2(Y\times Y,K)$.
Furthermore, let $S^2$ be the subspace of $H^4(X,K)$ which is given by products of elements in $S$.
By assumption (A1), this subspace contains $H^4(Y\times 0,K)$ and $H^4(0\times Y,K)$.
By the ring structure of the cohomology of $X$, it then follows that $H^2(Y\times Y,K)$ in (\ref{eq:thm0:H^2X}) is equal to the linear subspace of $H^2(X,K)$ that is spanned by those classes whose square lies in $S^2$.

By Lefschetz's theorem on $(1,1)$-classes, the cohomology of $Y^\sigma$ also satisfies (A1).
Hence, $H^2(Y^\sigma \times Y^\sigma,K)$ inside $SH^2(X^\sigma,K)^{\leq 4}$ has a similar intrinsic description as we have found for $H^2(Y\times Y,K)$ inside $SH^2(X,K)^{\leq 4}$.
This proves (\ref{eq1:claim:H2}).

In order to prove (\ref{eq2:claim:H2}) and (\ref{eq3:claim:H2}), we need the following Lemma, also used in \cite{charles,voisin:inv}.
In order to state it, we define for $i=1,\ldots ,5$ the following kernels:
\begin{align} \label{eq:F_i}
F_i:=\ker \left(\cup [D_i]:H^2(Y\times Y,K)\longrightarrow H^4(X,K)\right) .
\end{align}
Using Lemma \ref{lem:blow-up} and \ref{lem:mult}, we obtain the following Lemma, which is the analogue of Charles's Lemma 7 in \cite{charles}.

\begin{lemma} \label{lem7}
Using the identification (\ref{eq:H^2(YxY)}), the kernels $F_i\subseteq H^2(Y\times Y,K)$ are given as follows:
\begin{align}
F_1&= \left\{(0 ,\beta): \beta \in H^2(Y,K)\right\}, \label{eq1:lem7}\\
F_2&=\left\{(\beta , -\beta): \beta \in H^2(Y,K)\right\}, \label{eq2:lem7} \\
F_3&=\left\{(f^\ast\beta, -\beta): \beta \in H^2(Y,K)\right\}, \label{eq3:lem7} \\
F_4&=\left\{({f^\prime}^\ast\beta,-\beta): \beta \in H^2(Y,K)\right\}, \label{eq4:lem7} \\
F_5&=\left\{(\beta,0): \beta \in H^2(Y,K)\right\}. \label{eq5:lem7}
\end{align}
\end{lemma}

In addition to the above lemma, we have as in \cite{charles} the following.

\begin{lemma} \label{lem8}
Let $\alpha\in H^2(Y\times Y,K)$ be a non-zero class. 
Then the images of
\[
\cup\alpha,\cup h,\cup[D_1],\ldots ,\cup[D_5]:H^2(Y\times Y,K)\longrightarrow H^4(X,K)
\]
are in direct sum, $\cup h$ is injective and
\begin{align} \label{eq:ker(cupalpha)}
 \dim(\ker\cup \alpha)< b_2(Y) .
\end{align}
\end{lemma}

\begin{proof}
Apart from (\ref{eq:ker(cupalpha)}), the assertions in Lemma \ref{lem8} are immediate consequences of the ring structure of the cohomology of blow-ups, see Lemma \ref{lem:blow-up} and \ref{lem:mult}.

In order to proof (\ref{eq:ker(cupalpha)}), we write
\[
\alpha =\alpha_1+\alpha_2
\]
according to the decomposition (\ref{eq:H^2(YxY)}).
Without loss of generality, we assume $\alpha_1\neq 0$.
Then, $\cup\alpha$ restricted to $H^2(0\times Y,K)$ is injective.
Moreover, by Poincar\'e duality there is some $\beta_1\in H^2(Y\times 0,K)$ with 
\[
\beta_1 \cup \alpha_1 \neq 0
\]
Then, $\beta_1 \cup \alpha$ is nontrivial and does not lie in the image of $\cup\alpha$ restricted to $H^2(0\times Y,K)$.
Thus, $ \dim(\im (\cup \alpha))> b_2(Y)$ and (\ref{eq:ker(cupalpha)}) follows.
\end{proof}

Of course, the obvious analogues of Lemma \ref{lem7} and \ref{lem8} hold on $X^\sigma$.

Note the following elementary fact from linear algebra.
If a finite number of linear maps $l_1,\ldots ,l_r$ between two vector spaces have images in direct sum, then the kernel of a linear combination $\sum \lambda_il_i$ is given by intersection of all $\ker(l_i)$ with $\lambda_i\neq 0$.

By Lemma \ref{lem7}, each $F_i$ has dimension $b_2(Y)$ and hence the above linear algebra fact together with Lemma \ref{lem8} shows that there is a permutation $\rho \in \Sym(5)$ with
\[
\phi\left([D_i]\cdot K\right)=[D_{\rho(i)}^\sigma] \cdot K .
\]

We are now able to prove (\ref{eq3:claim:H2}).
For some real numbers $a_0,\ldots,a_5$ and for some class $\beta^\sigma \in H^2(Y^\sigma\times Y^\sigma,K)$ we have
\[
\phi(h)=a_0h^\sigma+\sum_{j=1}^5 a_j[D_j^\sigma]+\beta^\sigma .
\]
For $i=1,\ldots ,4$, the cup product $h\cup [D_i]$ vanishes and hence
\[
a_0h^\sigma\cup [D_{\rho(i)}^\sigma]+\sum_{j=1}^5 a_j[D_j^\sigma]\cup [D_{\rho(i)}^\sigma]+\beta^\sigma\cup [D_{\rho(i)}^\sigma]=0 .
\]
Since the cup product $[D_j^\sigma]\cup [D_k^\sigma]$ vanishes for $j\neq k$, we deduce 
\[
a_0h^\sigma\cup [D_{\rho(i)}^\sigma]+ a_{\rho(i)}[D_{\rho(i)}^\sigma]^2+\beta^\sigma\cup [D_{\rho(i)}^\sigma]=0 
\]
for all $i=1,\ldots ,4$.
From Lemma \ref{lem:blow-up}, it follows that $a_{\rho(i)}$ vanishes for all $i=1,\ldots 4$.

If $i$ is such that $\rho(i)\in \left\{1,\ldots 4\right\}$, then 
\[
h^\sigma\cup [D_{\rho(i)}^\sigma]=0\ \ \text{and so}\ \  \beta^\sigma\cup [D_{\rho(i)}^\sigma]=0 .
\]
By Lemma \ref{lem7}, the intersection $\bigcap_{j\neq k}F_j$ is zero for each $k=1,\ldots ,5$.
Since the same holds on $X^\sigma$, we deduce that $\beta^\sigma$ vanishes.
Hence,
\[
\phi(h)=a_0h^\sigma+a_{\rho(5)}[D_{\rho(5)}^\sigma] .
\]

In $H^4(X,K)$ we have the identity
\[
h\cup [D_5]= (i^\ast h) \cup [D_5] \in H^2(Y\times Y)\cup [D_5] ,
\]
and similarly on $X^\sigma$.
Since (\ref{eq1:claim:H2}) is already proven, we deduce
\[
a_0h^\sigma\cup [D_{\rho(5)}^\sigma]+a_{\rho(5)}[D_{\rho(5)}^\sigma]^2 \in H^2(Y^\sigma \times Y^\sigma )\cup [D_{\rho(5)}^\sigma] .
\]
This implies $a_{\rho(5)}=0$.
Since $\phi$ is an isomorphism, $a_0\neq 0$ follows, which proves (\ref{eq3:claim:H2}).

It remains to prove (\ref{eq2:claim:H2}).
That is, we need to see that $\rho\in \Sym(5)$ is the identity.
This will be achieved by a similar argument as in \cite[Lem.\ 11]{charles}.

Note that $h\cup[D_i]$ as well as $h^\sigma\cup[D_i^\sigma]$ vanish for $i\neq 5$ and are nontrivial for $i=5$.
Since (\ref{eq3:claim:H2}) is already proven, $\rho(5)=5$ follows.

By assumption on $Y$, $f^\ast$ and $f'^\ast$ fix $i^\ast h$.
Therefore, the intersection $F_2\cap F_3 \cap F_4$ is nontrivial.
Conversely, $F_1\cap F_i=0$ for all $i=2,3,4$.
Since analogue statements hold on $X^\sigma$, we obtain $\rho(1)=1$.

Next, we use that $F_i\oplus F_j = H^2(Y\times Y,K)$ for all $i=1,5$ and $j=2,3,4$.
This allows us to define for $2\leq j,k \leq 4$ endomorphisms $g_{j,k}$ of $F_1$ via the following composition:
\[
g_{j,k}:F_1 \hookrightarrow F_5\oplus F_j \overset{pr_1}{\longrightarrow} F_5 \hookrightarrow F_1\oplus F_k \overset{pr_1}{\longrightarrow} F_1 .
\]
There is a canonical identification between $F_1$ and $H^2(Y,K)$.
Using Lemma \ref{lem7}, a straightforward calculation then shows:
\begin{align} \label{eq:g_jk}
g_{3,2}=f^\ast,\ \ g_{4,2}=f'^\ast,\ \ g_{4,3}=(f'\circ f^{-1})^\ast , \ \ g_{j,j}=\id \ \ \text{and}\ \ g_{j,k}=g_{k,j}^{-1} ,
\end{align}
for all $2\leq j,k \leq 4$.

Analogue to (\ref{eq:F_i}), we define
\[
F_i^\sigma:=\ker\left(\cup[D_i^\sigma]:H^2(Y^\sigma\times Y^\sigma ,K) \longrightarrow H^4(X^\sigma,K)\right) .
\]
These subspaces are described by the corresponding statements of Lemma \ref{lem7}.
Thus, the above construction yields for any $2\leq j,k \leq 4$ endomorphisms $g_{j,k}^\sigma$ of $F_1^\sigma$.
Using the canonical identification of $F_1^\sigma$ with $H^2(Y^\sigma,K)$, these endomorphisms are given by
\begin{align} \label{eq:gsigma_jk}
g_{3,2}^\sigma=(f^\sigma)^\ast,\ \ g_{4,2}^\sigma=(f'^\sigma)^\ast,\ \ g_{4,3}^\sigma={(f'\circ f^{-1})^\sigma}^\ast , \ \ g_{j,j}^\sigma=\id \ \ \text{and}\ \ g_{j,k}^\sigma =(g_{k,j}^\sigma)^{-1} ,
\end{align}
for all $2\leq j,k \leq 4$.

Since $\phi$ maps $[D_1]$ to a multiple of $[D_1^\sigma]$, it follows that the restriction of $\phi$ to $F_1$ induces a $K$-linear isomorphism 
\begin{align} \label{eq:psi:H2}
\psi: F_1=H^2(Y,K)\stackrel{\sim}\longrightarrow  H^2(Y^\sigma,K)=F_1^\sigma .
\end{align}
Since $\phi$ maps $F_i$ isomorphically to $F_{\rho(i)}^\sigma$, the above isomorphism satisfies
\begin{align} \label{eq:psi_g_jk}
\psi\circ g_{j,k}=g_{\rho(j),\rho(k)}^\sigma \circ \psi
\end{align}
for all $2\leq j,k \leq 4$.

We now denote the eigenvalues of $g_{j,k}$ by $\Eig(g_{j,k})$, and similarly for $g_{j,k}^\sigma$.
Since $f$ and $f'$ are automorphisms, it follows from (A2) and (\ref{eq:g_jk}) that $\Eig(g_{3,2})$ and $\Eig(g_{4,2})$ are distinct $\Aut(\C)$-invariant sets of roots of unity.
By Lemma \ref{lem:eigenvalues} and since $g_{j,k}=g_{k,j}^{-1}$, we deduce:
\begin{align*}
\Eig(g_{3,2})=\Eig(g_{2,3})=\Eig(g_{3,2}^\sigma)=\Eig(g_{2,3}^\sigma) ,\\ 
\Eig(g_{4,2})=\Eig(g_{2,4})=\Eig(g_{4,2}^\sigma)=\Eig(g_{2,4}^\sigma).
\end{align*} 
Since $g_{4,3}=g_{2,3}\circ g_{4,2}$ and $g_{3,4}=g_{2,4}\circ g_{3,2}$, it also follows that each of the sets $\Eig(g_{3,4})$, $\Eig(g_{4,3})$, $\Eig(g_{3,4}^\sigma)$ and $\Eig(g_{4,3}^\sigma)$ is distinct from $\Eig(g_{2,3})$ and $\Eig(g_{4,2})$.
Therefore, (\ref{eq:psi_g_jk}) implies that $\rho$ respects the subsets $\left\{2,3\right\}$ and $\left\{2,4\right\}$.
Hence, $\rho=\id$, as we wanted.
This finishes the proof of Claim \ref{claim:H2}.
\end{proof}

Since $b_1(Y)=0$ and $\dim(Y)=2$, the cohomology algebra $H^\ast(0\times Y,K)$ is a subalgebra of $SH^2(X,K)^{\leq 4}$.
Restriction of $\phi$ therefore extends the $K$-linear isomorphism $\psi$ from (\ref{eq:psi:H2}) to an isomorphism
\begin{align} \label{eq:psi}
\psi: H^\ast(Y,K)\stackrel{\sim}\longrightarrow  H^\ast(Y^\sigma,K)
\end{align}
of graded $K$-algebras which we denote with the same letter.
Since $\rho$ in the proof of Claim \ref{claim:H2} is the identity, it follows from (\ref{eq:g_jk}), (\ref{eq:gsigma_jk}) and (\ref{eq:psi_g_jk}) that $\psi$ satisfies (P2).

In order to prove (P1), we note that
\begin{align*} 
\ker\left(\cup [D_5]: F_1 \oplus h\cdot K \longrightarrow H^4(X,K)\right)=(i^\ast h-h)\cdot K ,
\end{align*}
where $i^\ast h\in F_1=H^2(0\times Y,K)$.
A similar statement holds on $X^\sigma$.
Since $\phi$ maps $F_1$ to $F_1^\sigma$, $[D_5]\cdot K$ to $[D_5^\sigma]\cdot K$ and $h\cdot K$ to $h^{\sigma}\cdot K$, it follows that $\phi$ maps $i^\ast h\cdot K$ to ${i^\sigma}^\ast h^{\sigma} \cdot K$.
This finishes the proof of Proposition \ref{prop:psi}.

\subsection{Proof of Proposition \ref{prop:psi:(1,1)}} \label{subsec:proof:prop:psi:(1,1)}
As in the proof of Proposition \ref{prop:psi}, we use (\ref{eq:Lef}) in order to identify classes of degree $\leq n$ on $T$ with classes on $X$.
Further, $SH^{1,1}(-,K)$ denotes the subalgebra of $H^\ast(-,K)$ that is generated by $H^{1,1}(-,K)$; its quotient by elements of degree $\geq r+1$ is denoted by $SH^{1,1}(-,K)^{\leq r}$.

Let us now suppose that there is a $K$-linear isomorphism
\begin{align}
\phi': H^{1,1}(T,K)\stackrel{\sim} \longrightarrow H^{1,1}(T^\sigma , K) ,
\end{align}
which induces a weak isomorphism between the respective intersection forms.
Then we have the following analogue of Claim \ref{claim:phi} in the proof of Proposition \ref{prop:psi}:

\begin{claim} \label{claim:xi}
The isomorphism from (\ref{eq1:prop:psi}) induces a unique isomorphism 
\[
\phi:SH^{1,1}(X,K)^{\leq 4} \stackrel{\sim}\longrightarrow SH^{1,1}(X^\sigma, K)^{\leq 4}
\] 
of graded $K$-algebras.
\end{claim}

\begin{proof}
As in the proof of Claim \ref{claim:phi}, this claim reduces to showing the following:
Suppose we have $K$-rational $(1,1)$-classes $\alpha_1,\ldots ,\alpha_r$ and $\beta_1,\ldots, \beta_r$ on $T$ such that
\begin{align} \label{eq:claim:xi}
\sum_i \phi'(\alpha_i)\cup \phi'(\beta_i)\cup \eta \cup [T^\sigma]=0 \ \ \text{in}\ \ H^{2N+8}(X^\sigma,K) ,
\end{align}
for all $\eta\in SH^{1,1}(X^\sigma,K)^{2n-4}$.
Then, $\sum_i \phi'(\alpha_i)\cup \phi'(\beta_i)$ vanishes.

In order to prove the latter, let $\omega$ be the hyperplane class on $X^\sigma$ with $[T^{\sigma}]=\omega^{N+4-n}$.
With respect to this Kähler class we obtain a decomposition into primitive pieces:
\[
\sum_i \phi'(\alpha_i)\cup \phi'(\beta_i)= \delta_0\cdot \omega^2+\delta_1\cup \omega +\delta_2 ,
\] 
where $\delta_j \in H^{j,j}(X,\C)_{pr}$.
Since $\omega$ is an integral class, it follows that $\delta_j$ lies in $H^{j,j}(X,K)_{pr}$.
The above identity then shows $\delta_2 \in SH^{1,1}(X,K)$.

At this point, we use the assumption in Proposition \ref{prop:psi:(1,1)} which ensures that $K$ is stable under complex conjugation.
Indeed, this assumption allows us to choose for $j=0,1,2$ the following $K$-rational classes:
\[
\eta_j:= \overline {\delta_j} \cdot \omega^{n-2-j} \in SH^{1,1}(X^\sigma,K)^{2n-4} . 
\]
For $j=0,1,2$, we put $\eta=\eta_j$ in (\ref{eq:claim:xi}).
Then, the Hodge--Riemann bilinear relations yield $\delta_j=0$ for $j=0,1,2$. 
This finishes the proof of Claim \ref{claim:xi}.
\end{proof}

Exploiting the isomorphism of $K$-algebras $\phi$ from Claim \ref{claim:xi}, the proof of Proposition \ref{prop:psi:(1,1)} is now obtained by changing the notation in the corresponding part of the proof of Proposition \ref{prop:psi}.
This finishes the proof of Proposition \ref{prop:psi:(1,1)}.

\section{Some simply connected surfaces with special automorphisms} \label{sec:surfaces}
In this section we construct for any integer $g\geq 1$ a simply connected surface $Y_{g}$ of geometric genus $g$ and with special automorphisms.
In the proof of Theorem \ref{thm0} in Section \ref{sec:thm0'}, we will then apply the construction from Section \ref{sec:charles} to these surfaces.
In Section \ref{sec:thm2}, we will use the examples from Section \ref{sec:thm0'} in order to prove Theorem \ref{thm1}.
It is only the proof of the latter theorem where it will become important that $b_2(Y_g)$ tends to infinity if $g$ does.

\subsection{Hyperelliptic curves with special automorphisms} \label{subsec:hyperell}
For $g\geq 1$, let $C_g$ denote the hyperelliptic curve with affine equation $y^2=x^{2g+1}-1$, see \cite{tashiro}.
The complement of this affine piece in $C_g$ is a single point which we denote by $\infty$.
For a primitive $(2g+1)$-th root of unity $\zeta_{2g+1}$, the maps
\[
(x,y)\mapsto (\zeta_{2g+1}\cdot x,y)\ \ \text{and}\ \ (x,y)\mapsto (x,-y)
\]
induce automorphisms of $C_g$ which we denote by $\eta_g$ and $\iota$ respectively.
Then, $\iota$ has the $(2g+2)$ fixed points
\[
(1,0),(\zeta_{2g+1},0),\ldots ,(\zeta_{2g+1}^{2g},0)\ \text{and}\ \infty .
\]
The automorphism $\eta_g$ fixes $\infty$ and performs a cyclic permutation on the remaining fixed points. 
The corresponding permutation matrix has eigenvalues $1,\zeta_{2g+1},\ldots ,\zeta^{2g}_{2g+1}$.

The holomorphic $1$-forms 
\[
\frac{x^{i-1}}{y}\cdot dx ,
\]
where $i=1,\ldots ,g$, form a basis of $H^{1,0}(C_g)$.
Therefore, $\eta_g^\ast$ has eigenvalues $\zeta_{2g+1},\ldots ,\zeta_{2g+1}^g$ on $H^{1,0}(C_g)$.
Moreover, $\iota$ acts on $H^1(C_g,\Z)$ by multiplication with $-1$.

\subsection{The elliptic curve $E_i$} \label{subsec:ell} 
Let $E_i$ be the elliptic curve $\C/(\Z\oplus i\Z)$, cf.\ Section \ref{subsec:j-inv}. 
Multiplication by $i$ and $-1$ induces automorphisms $\eta_i$ and $\iota$ of $E_i$ respectively.
The involution $\iota$ has four fixed points.
The action of $\eta_i$ fixes two of those fixed points and interchanges the remaining two.
On $H^{1,0}(E_i)$, the automorphisms $\iota$ and $\eta_i$ act by multiplication with $-1$ and $i$ respectively. 

\subsection{Products modulo the diagonal involution} \label{subsec:surfaces} 
For $g\geq 1$, we consider the product $C_g\times E_i$, where $C_g$ and $E_i$ are defined above.
On this product, the involution $\iota$ acts via the diagonal.
This action has $8g+8$ fixed points.
Let $\widetilde {C_g\times E_i}$ be the blow-up of these fixed points.
Then,
\begin{align} \label{eq:def:Y}
Y_g:=\widetilde {C_g\times E_i}/\iota 
\end{align}
is a smooth surface. 
For instance, $Y_1=K3(C_1\times E_i)$ is a Kummer K3 surface, see Section \ref{subsec:K3}.

\begin{lemma} \label{lem:sconnected}
The surface $Y_g$ is simply connected.
\end{lemma}
\begin{proof}
It suffices to prove that the normal surface
\[
Y_g':=(C_g\times E_i)/\iota 
\]
is simply connected.
Projection to the second coordinate induces a map
\[
\pi: Y_g' \longrightarrow \CP^1 .
\]
Let $U\subseteq \CP^1$ be the complement of the $4$ branch points of $E_i\to \CP^1$.
Then, restriction of $\pi$ to $V:=\pi^{-1}(U)$ yields a fiber bundle $\pi|_V:V\to U$ with fiber $C_g$.
Since $U$ is homotopic to a wedge of $3$ circles, the long exact homotopy sequence yields a short exact sequence
\[
0\longrightarrow \pi_1(C_g)\longrightarrow \pi_1(V) \longrightarrow \pi_1(U) \longrightarrow 0 \ .
\]
Since $\pi$ has a section, this sequence splits.
Since $V$ is the complement of a divisor in $Y_g'$, the natural map $\pi_1(V)\to \pi_1(Y_g')$ is surjective, see \cite[Prop.\ 2.10]{kollar}.
Therefore, the above split exact sequence shows that $\pi_1(Y_g')$ is generated by the fundamental group of a general fiber together with the image of the fundamental group of a section of $\pi$.
The latter is clearly trivial.
Furthermore, the inclusion of a general fiber $C_g\hookrightarrow Y_g'$ is homotopic to the inclusion of a special fiber $C_g/\iota \cong \CP^1$, which is simply connected.
It follows that the image of $\pi_1(C_g)\to \pi_1(Y_g')$ is trivial.
This proves the lemma.
\end{proof}

\begin{definition} \label{def:ff'}
Let $Y_g$ be as in (\ref{eq:def:Y}).
Then we define the automorphisms $f$ and $f'$ of $Y_g$ to be induced by $\eta_g\times \id$ and $\id \times \eta_i$ respectively.
\end{definition}

\begin{lemma} \label{lem:Y}
The surface $Y_g$ with automorphisms $f$ and $f'$ as above satisfies (A1)--(A3).
\end{lemma}

\begin{proof}
In order to describe the second cohomology of $Y_g$, we denote the exceptional $\CP^1$-curves of $Y_g$ by $D_1,\ldots ,D_{8g+8}$.
Then, for any field $K$:
\begin{align} \label{eq:H2Y}
H^2(Y_g,K)= H^2(C_{g}\times E_i,K)\oplus \left(\bigoplus_{i=1}^{8g+8} [D_i]\cdot K\right) .
\end{align}
It follows from the discussion in Section \ref{subsec:hyperell} (resp.\ \ref{subsec:ell}) that the action of $f$ (resp.\ $f'$) on $H^2(Y_g,\C)$ has eigenvalues $1,\zeta_{2g+1},\ldots ,\zeta^{2g}_{2g+1}$ (resp.\ $\pm 1,\pm i$).
Moreover, the same statement holds for their actions on $H^{1,1}(Y_g,\C)$.
This proves (A2) and (A3).

By (\ref{eq:H2Y}), nontrivial rational $(1,1)$-classes on $C_g$ and $E_i$ induce classes $\alpha$ and $\beta$ in $H^{1,1}(Y_g,\Q)$ which satisfy (A1).
This finishes the prove of the lemma.
\end{proof}

\section{Multilinear intersection forms on $H^2(-,\R)$ and $H^{1,1}(-\C)$} \label{sec:thm0'} 
Here we prove Theorem \ref{thm0}.
This will be achieved by Lemma \ref{lem:Tgnsimple} and Theorem \ref{thm0'} below, where more precise statements are proven.

Let $n\geq 4$ and $g\geq 1$. 
Moreover, let $Y_g$ be the simply connected surface with automorphisms $f$ and $f'$ from Definition \ref{def:ff'}.
We pick an ample divisor on $Y_g$ which is fixed by $f$ and $f'$.
A sufficiently large multiple of this divisor gives an embedding
\[
i:Y_g\hookrightarrow \CP^N
\]
with $n\leq N+4$ such that the actions of $f$ and $f'$ fix the pullback of the hyperplane class.

Next, let 
\[
X_g:=Bl_{Z_1\cup\ldots \cup Z_5} \left(Y_g\times Y_g\times \CP^N\right) 
\]
be the blow-up of $Y_g\times Y_g\times \CP^N$ along $Z_1\cup \ldots \cup Z_5$, where $Z_i$ is defined in (\ref{eq:def:Z_i}).
Since $n\leq N+4$, $X_g$ contains a smooth $n$-dimensional complete intersection subvariety 
\begin{align} \label{eq:def:Tg}
T_{g,n}\subseteq X_g .
\end{align}
Since $Y_g$, $f$ and $f'$ are defined over $\Q[\zeta_{8g+4}]=\Q[\zeta_{2g+1},i]$, so is $X_g$ and we may assume that the same holds true for $T_{g,n}$.

\begin{lemma} \label{lem:Tgnsimple}
Let $n\geq 2$, then the variety $T_{g,n}$ from (\ref{eq:def:Tg}), as well as each of its conjugates, is simply connected.
\end{lemma}

\begin{proof}
Since $Y_g$ is simply connected by Lemma \ref{lem:sconnected}, so is $X_g$. 
By the Lefschetz hyperplane theorem, $T_{g,n}$ is then simply connected for $n\geq 2$.

Since the curves $C_g$ and $E_i$ in the definition of $Y_g$ are defined over $\Z$, it follows that $Y_g$ is isomorphic to any conjugate $Y_g^\sigma$.
Thus, $Y_g^\sigma$ is simply connected and the above reasoning shows that the same holds true for $T_{g,n}^\sigma$, as long as $n\geq 2$.
This proves the lemma.
\end{proof}

The next theorem, which implies Theorem \ref{thm0} from the introduction, shows that certain automorphisms $\sigma\in \Aut(\C)$ which act nontrivially on $\Q[\zeta_{8g+4}]$ change the analytic topology as well as the complex Hodge structure of $T_{g,n}$.

\begin{theorem} \label{thm0'}
Let $g\geq 1$ and $n\geq 4$ be integers and let $\sigma\in \Aut(\C)$ with $\sigma(i)=i$ and $\sigma(\zeta_{2g+1})\neq \zeta_{2g+1}$ or vice versa.
Then, the $\R$-multilinear intersection forms on $H^2(T_{g,n},\R)$ and $H^2(T_{g,n}^\sigma,\R)$, as well as the $\C$-multilinear intersection forms on $H^{1,1}(T_{g,n},\C)$ and $H^{1,1}(T_{g,n}^\sigma,\C)$, are not weakly isomorphic.
\end{theorem}

\begin{proof}
For ease of notation, we assume $\sigma(i)=i$ and $\sigma(\zeta_{2g+1})=\zeta_{2g+1}^{-1}$. 
The general case is proven similarly.

Since the curves $C_g$ and $E_i$ from Sections \ref{subsec:hyperell} and \ref{subsec:ell} are defined over $\Z$, it follows that the isomorphism type of $Y_g$ is invariant under any automorphism of $\C$.
Hence, we may identify $Y_g$ with $Y_g^\sigma$. 
Under this identification, $f'^\sigma=f'$ since $i$ is fixed by $\sigma$.
Moreover, $f^\sigma=f^{-1}$, since it is induced by the automorphism 
\[
\eta_g^{-1}\times \id \in \Aut (C_{g}\times E_i) .
\]

Suppose that the $\R$-multilinear intersection forms on $H^2(T_{g,n},\R)$ and $H^2(T_{g,n}^\sigma,\R)$ are weakly isomorphic.
By Lemma \ref{lem:Y}, Proposition \ref{prop:psi} applies and we obtain an $\R$-algebra automorphism of $H^\ast(Y_g,\R)$ with properties (P1) and (P2).
By (P1), 
\[
\psi(i^\ast h)=b\cdot i^\ast h
\] 
for some $b\in \R^\times$.
Since the square of $i^\ast h$ generates $H^4(Y_g,\R)$, it follows that in degree $4$, the automorphism $\psi$ is given by multiplication with a positive real number.

We extend $\psi$ now $\C$-linearly and obtain an automorphism
\[
\psi:H^\ast(Y_g,\C) \stackrel{\sim}\longrightarrow H^\ast(Y_g,\C) ,
\]
which we denote by the same letter and which satisfies
\begin{align} \label{eq:f&f'}
\psi\circ f = f^{-1}\circ \psi\ \ \text{and} \ \ \psi\circ f'=f'\circ \psi .
\end{align}

Let us now pick nontrivial classes $\omega \in H^{1,0}(C_g)$ and $\omega'\in H^{1,0}(E_i)$ with $\eta_g^\ast\omega=\zeta_{2g+1}\cdot \omega$ and $\eta_i^\ast\omega'=i\cdot \omega'$.
Then, $\omega\cup \overline {\omega'}$ lies in $H^{1,1}(Y_g)$ and we consider $\psi(\omega\cup \overline{\omega'})$ in $H^2(Y_g,\C)$.
By (\ref{eq:f&f'}), $f^{-1}$ and $f'$ act on this class by multiplication with $\zeta_{2g+1}$ and $-i$ respectively.
We claim that the only class in $H^2(Y_g,\C)$ with this property is $\overline{\omega}\cup \overline{\omega'}$ and so
\begin{align} \label{eq:psi(omegacupomega')}
\psi(\omega\cup \overline{\omega'})=\lambda\cdot \overline{\omega}\cup \overline{\omega'} 
\end{align}
for some non-zero $\lambda\in \C$.
Indeed, since $\eta_i$ interchanges two of the fixed points of $\iota$ on $E_i$ and fixes the remaining two, $f'^\ast$ has eigenvalues $\pm 1$ on the subspace of exceptional divisors in (\ref{eq:H2Y}).
Therefore, $\psi(\omega\cup \overline{\omega'})$ needs to be contained in $H^2(C_g\times E_i,\C)$.
On this subspace, ${f^{-1}}^\ast$ and $f'^\ast$ are given by $(\eta_g^{-1}\times \id)^\ast$ and $(\id \times \eta_i)^\ast$ respectively.
Our claim follows by the explicit description of $\eta_g$ and $\eta_i$ in Sections \ref{subsec:hyperell} and \ref{subsec:ell}.

Together with its complex conjugate, equation (\ref{eq:psi(omegacupomega')}) shows:
\[
\psi(\omega \cup \overline{\omega'} \cup \overline {\omega} \cup \omega')= -|\lambda|^2\cdot \omega \cup \overline{\omega'} \cup \overline {\omega} \cup \omega'  .
\]
Since the above degree four class generates $H^4(Y_g,\C)$, we deduce that $\psi$ is given in degree four by multiplication with $-|\lambda|^2$.
As we have seen earlier, this number should be positive, which is a contradiction.
This finishes the proof of the first assertion in Theorem \ref{thm0'}.

For the proof of the second assertion, assume that the $\C$-multilinear intersection forms on $H^{1,1}(T_{g,n},\C)$ and $H^{1,1}(T_{g,n}^\sigma,\C)$ are weakly isomorphic.
By Lemma \ref{lem:Y} and Proposition \ref{prop:psi}, this yields an automorphism $\psi$ of $H^{1,1}(Y_g,\C)$ which satisfies (\ref{eq:f&f'}).
Then, $f^{-1}$ and $f'$ act on $\psi(\omega\cup \overline{\omega'})$ by multiplication with $\zeta_{2g+1}$ and $-i$ respectively.
This is a contradiction, since $H^{1,1}(Y_g,\C)$ does not contain such a class. 
This finishes the proof of the theorem.
\end{proof}

Recall from (\ref{eq:def:Tg}) that $T_{g,n}$ is defined over the cyclotomic number field $\Q[\zeta_{8g+4}]$.
This number field contains the totally real subfield
\[
K_g:=\Q[\zeta_{8g+4}+\zeta_{8g+4}^{-1}].
\]
For instance, $K_1=\Q[\sqrt 3]$.
From Theorem \ref{thm0'}, we deduce the following

\begin{corollary} \label{cor:Kg}
Let $K_g\subseteq K\subseteq \C$ be fields, and let $\sigma\in \Aut(\C)$ with $\sigma(i)=i$ and $\sigma(\zeta_{2g+1})\neq \zeta_{2g+1}$ or vice versa.
Then the intersection forms on the equidimensional vector spaces $H^{1,1}(T_{g,n},K)$ and $H^{1,1}(T_{g,n}^\sigma,K)$ are not weakly isomorphic.
\end{corollary}

\begin{proof} 
By Theorem \ref{thm0'} it suffices to prove that the $(1,1)$-classes on $T_{g,n}$ are spanned by $K_g$-rational ones.
Modulo divisor classes, $H^{1,1}(T_{g,n})$ is given by $H^{1,1}(Y_g)\oplus H^{1,1}(Y_g)$.
Furthermore, modulo divisors, $H^{1,1}(Y_g)$ is given by the $\iota$-invariant classes on $E_i\times C_g$.
The complex Hodge structure of $E_i$ and $C_g$ is generated by $\Q[i]$- and $\Q[\zeta_{2g+1}]$-rational classes respectively, see \cite{tashiro} for the latter.
We may now arrange that the induced generators of $H^{1,1}(Y_g)$ are invariant under complex conjugation and thus lie in the subspace of $K_g$-rational classes.
This concludes the proof of the corollary.
\end{proof}

\begin{remark}  \label{rem:Hodgeconj}
Our types of arguments are consistent with Conjecture \ref{conj1} in the sense that they cannot detect conjugate varieties with non-isomorphic algebras of $\Q$-rational $(p,p)$-classes.
This is because the essential ingredient in the proof of Theorem \ref{thm0'} is a variety $Y$ with an automorphism whose action on $H^{p,p}(Y,K)$ has a set of eigenvalues which is not $\Aut(\C)$-invariant.
(In our arguments, this role is played by the surface $Y_g$ with the automorphism $f\circ f'$.)
For $K=\Q$, the characteristic polynomial of such an action has rational coefficients and so the above situation cannot happen.
\end{remark}

\begin{remark} \label{rem:3-folds?}
Using Freedman's classification of simply connected topological $4$-manifolds, one can prove that simply connected conjugated smooth complex projective surfaces are always homeomorphic.
On the other hand, Theorem \ref{thm0} shows that in any dimension at least $4$, there are simply connected conjugate smooth complex projective varieties which are not homeomorphic.
The case of dimension three remains open.
\end{remark}

\section{Non-homeomorphic conjugate varieties in each birational equivalence class} \label{sec:thm2}
In this section we prove Theorem \ref{thm1}.
For this purpose, let $Z$ be a given smooth complex projective variety of dimension $\geq 10$. 
Next, let $T_{g,4}$ be the four-dimensional smooth complex projective variety, defined in (\ref{eq:def:Tg}).
By (\ref{eq:thm0:H^2X}) and (\ref{eq:H2Y}), the second Betti number of $T_{g,4}$ equals $24g+26$.
We may therefore choose an integer $g\geq 1$ with
\begin{align} \label{eq:b2Tg}
b_2(T_{g,4}) > b_4(Z) +4 . 
\end{align}

From some projective space, $Z$ is cut out by finitely many homogeneous polynomials.
We denote the field extension of $\Q$ which is generated by the coefficients of these polynomials by $L$.
Since $L$ is finitely generated, and after possibly replacing $g$ by a suitable larger integer, we may pick an automorphism $\sigma$ of $\C$ which fixes $L$ and $i$ but not $\zeta_{2g+1}$.

Since $T_{g,4}$ has dimension $4$, it can be embedded into $\CP^9$.
The assumption $\dim(Z)\geq 10$ therefore ensures that we may fix an embedding of $T_{g,4}$ into the exceptional divisor of the blow-up $\hat Z$ of $Z$ in a point $p\in Z$.
We then define the following element in the birational equivalence class of $Z$:
\begin{align} \label{def:W}
W:=Bl_{T_{g,4}}(\hat Z) .
\end{align}

Since conjugation commutes with blow-ups, the $\sigma$-conjugate of $W$ is given by
\begin{align} \label{def:Wsigma}
W^\sigma=Bl_{T_{g,4}^\sigma}(\hat Z^\sigma) ,
\end{align}
where $\hat Z^\sigma$ is the blow-up of $Z^\sigma$ in a point $p^\sigma\in Z^\sigma$ and $T_{g,4}^\sigma$ is embedded in the exceptional divisor of this blow-up.
Since $\sigma$ fixes $L$, we have $Z^\sigma \cong Z$.
Therefore, $W$ and $W^\sigma$ are both birational to $Z$.
Hence, Theorem \ref{thm1} follows from the following result.

\begin{theorem} \label{thm2}
Let $W$ and $\sigma$ be as above.
Then the graded even-degree real cohomology algebras of $W$ and $W^\sigma$ are non-isomorphic.
\end{theorem}

\begin{proof}
For a contradiction, let us assume that there is an isomorphism
\[
\gamma:H^{2\ast}(W,\R)\longrightarrow  H^{2\ast}(W^\sigma,\R)
\]
of graded $\R$-algebras.  
Using pullbacks, we regard $H^{2\ast}(Z,\R)\subseteq H^{2\ast}(\hat Z,\R)$ and $H^{2\ast}(Z^\sigma,\R)\subseteq H^{2\ast}(\hat Z^\sigma,\R)$ as subalgebras of $H^{2\ast}(W,\R)$ and $H^{2\ast}(W^\sigma,\R)$ respectively.
By Lemma \ref{lem:blow-up}, 
\begin{align}
\label{eq:H^2W}
H^2(W,\R)&=H^2(Z,\R)\oplus [H]\cdot \R\oplus [D]\cdot \R , \\
\label{eq:H^2Wsigma}
H^2(W^\sigma,\R)&=H^2(Z^\sigma,\R)\oplus [H^\sigma]\cdot \R\oplus [D^\sigma] \cdot \R ,
\end{align}
where $H\subset \hat Z$ and $H^\sigma \subset \hat Z^\sigma$ are the exceptional divisors above the blown-up points, and
\[
j:D\hookrightarrow W\ \ \ \text{and}\ \ \ j^\sigma :D^\sigma \hookrightarrow W^\sigma
\] 
are the exceptional divisors of the blow-ups along $T_{g,4}$ and $T_{g,4}^\sigma$ respectively.

Any cohomology class of positive degree on $Z$ is Poincar\'e dual to a homology class which does not meet the center of the blow-up $\hat Z\to Z$.
This shows that for any $\eta\in H^k(Z,\R)$, with $k\geq 1$, and for any $\alpha \in H^\ast(D,\R)$, 
\[ 
\eta\cup [H]=0\ \ \text{and}\ \ \ \eta\cup j_\ast(\alpha)=0 .
\]
A similar statement holds on $W^\sigma$ and we will use these properties tacitly. 

The restriction of $-[H]$ to $H\subset \hat Z$ is given by $c_1(\mathcal O_H(1))$; its restriction to $T_{g,4}$ is therefore ample.
By Lemma \ref{lem:blow-up}, we have 
\[
b_4(W)=b_4( Z)+b_2(T_{g,4}) + 2 .
\]
It then follows from (\ref{eq:b2Tg}) that the second primitive Betti number of $T_{g,4}$ is bigger than $b_4(W)/2$. 
Since $T_{g,4}$ is four-dimensional, and since $-[H]$ restricts to an ample class on $T_{g,4}$, it follows that $H^2(Z,\R)\oplus [H]\cdot \R$ inside $H^2(W,\R)$ is given by those classes whose multiplication on $H^4(W,\R)$ has kernel of dimension bigger than $b_4(W)/2$.
A similar statement holds for $H^2(Z^\sigma,\R)\oplus [H^\sigma]\cdot \R$ inside $H^2(W^\sigma,\R)$ and so $\gamma$ needs to take $H^2(Z,\R)\oplus [H]\cdot \R$ to $H^2(Z^\sigma,\R)\oplus [H^\sigma]\cdot \R$.
Since $\gamma$ is an isomorphism, it follows that
\begin{align} \label{eq:gamma(D)}
\gamma([D])= \alpha^\sigma +a\cdot [H^\sigma]+b\cdot [D^\sigma]
\end{align}
holds for some $\alpha^\sigma\in H^2(Z^\sigma,\R)$ and $b\neq 0$.

Cup product with $[D]$ on $H^2(W,\R)$ has two-dimensional image, spanned by $[D]\cup [H]$ and $[D]^2$.
For any $\beta^\sigma \in H^2(Z^\sigma,\R)$, the following classes are therefore linearly dependent:
\[
\gamma([D])\cup \beta^\sigma,\ \ \gamma([D])\cup [H^\sigma] \ \ \text{and}\ \ \gamma([D])\cup [D^\sigma] .
\] 
Since $b\neq 0$, this is only possible if $\alpha^\sigma\cup \beta^\sigma=0$ for all $\beta^\sigma$.
Hence, $\alpha^\sigma=0$.

Since $\alpha^\sigma=0$, it follows from $[D]\cup [H]\neq 0$ that $\gamma([H])\in H^2(Z^\sigma,\R)\oplus [H^\sigma]\cdot \R$ cannot be contained in $H^2(Z^\sigma,\R)$ and hence
\[
\gamma([H])=\tilde \alpha^\sigma + c\cdot [H^\sigma]
\]
for some $\tilde \alpha^\sigma \in H^2(Z^\sigma,\R)$ and $c\neq 0$.
As cup product with $[H]$ on $H^2(W,\R)$ has two-dimensional image, the above argument which showed $\alpha^\sigma=0$, also implies $\tilde \alpha^\sigma=0$.
Thus, $\gamma$ takes $[H]\cdot \R$ to $[H^\sigma]\cdot \R$.
It follows that $\gamma$ takes $H^2(Z,\R)$ to $H^2(Z^\sigma,\R)$, since these are the kernels of cup product with $[H]$ and $[H^\sigma]$ respectively.

Since $T_{g,4}$ is four-dimensional, we have $[H]^5\cup [D]=0$.
Then application of $\gamma$ yields:
\[
c^5 \cdot[H^\sigma]^5\cup (a\cdot [H^\sigma]+b\cdot [D^\sigma])=0 .
\]
Since $[H^\sigma]^5\cup [D^\sigma]$ vanishes, whereas $[H^\sigma]^6$ is nontrivial, it follows from $c\neq 0$ that $a$ vanishes.
Thus, $\gamma$ maps $[D]\cdot \R$ to $[D^\sigma]\cdot \R$ and we conclude that $\gamma$ respects the decompositions (\ref{eq:H^2W}) and (\ref{eq:H^2Wsigma}).

The latter implies that $\gamma$ induces an $\R$-linear isomorphism between the ideals $([D])\subseteq H^{2\ast}(W,\R)$ and $([D^\sigma])\subseteq H^{2\ast}(W^\sigma,\R)$.
In order to state the key-property of this isomorphism, we identify cohomology classes on $T_{g,4}$ and $T_{g,4}^\sigma$ with their pullbacks to the exceptional divisors $D$ and $D^\sigma$ respectively. 

\begin{lemma} \label{lem:thm2}
For every $\alpha\in H^{2k}(T_{g,4},\R)$, there exists a unique $\alpha^\sigma\in H^{2k}(T_{g,4}^\sigma,\R)$ such that
\[
\gamma([D]\cup j_\ast (\alpha))=[D^\sigma]\cup j_\ast^\sigma(\alpha^\sigma) .
\]
\end{lemma}

\begin{proof}
For $0\leq k\leq 2$, let us fix some $\alpha\in H^{2k}(T_{g,4},\R)$ and note that
\[
H^{2k+2}(W^\sigma,\R)=H^{2k+2}(Z^\sigma,\R)\oplus [H^\sigma]^{k+1}\cdot \R \oplus j^\sigma_{\ast} (H^{2k}(D^\sigma,\R))  .
\]
Since $\gamma$ maps $[D]$ to a multiple of $[D^\sigma]$, and since products of $[D^\sigma]$ with positive-degree classes on $Z^\sigma$ always vanish, the above identity shows
\[
\gamma([D]\cup j_\ast (\alpha))=[D^\sigma]\cup j_\ast^\sigma(\alpha^\sigma) +e\cdot [D^\sigma]\cup [H^\sigma]^{k+1} \ ,
\]
for some $\alpha^\sigma \in H^{2k}(D^\sigma,\R)$ and $e\in \R$.
The restrictions of $-[H]$ to $T_{g,4}$ and $-[H^\sigma]$ to $T_{g,4}^\sigma$ are ample classes $\omega \in H^2(T_{g,4},\R)$ and $\omega^\sigma\in H^2(T_{g,4}^\sigma,\R)$ respectively.

Now suppose that $\alpha$ in the above formula is primitive with respect to $\omega$.
Then the cup product of the above class 
with $\gamma([H])^{5-2k}$ vanishes.
Since $\gamma([H])$ is a multiple of $[H^\sigma]$, 
\[
[D^\sigma]\cup j_\ast^\sigma(\alpha^\sigma\cup (\omega^\sigma)^{5-2k}) +e\cdot(-1)^{k+1} j^\sigma_\ast((\omega^\sigma)^{6-k})=0 .
\]
This implies firstly that $e=0$ and secondly that $\alpha^\sigma\cup (\omega^\sigma)^{5-2k}$ vanishes as class on $D^\sigma$.
By the Hard Lefschetz Theorem, the latter already implies that $\alpha^\sigma$, which a priori is only a class on $D^\sigma$, is in fact a primitive class on $T_{g,4}^\sigma$. 

For arbitrary $\alpha\in H^k(T_{g,4},\R)$, the existence of $\alpha^\sigma$ now follows -- since $\gamma$ takes $[H]\cdot \R$ to $[H^\sigma]\cdot \R$ -- from the Lefschetz decompositions with respect to $\omega$ and $\omega^\sigma$; 
the uniqueness is immediate from Lemma \ref{lem:blow-up}.
This concludes Lemma \ref{lem:thm2}.
\end{proof}

By Lemma \ref{lem:thm2}, we are now able to define an $\R$-linear map
\[
\phi:H^{2\ast}(T_{g,4},\R)\longrightarrow  H^{2\ast}(T_{g,4}^\sigma,\R) ,
\]
by requiring
\[
\gamma ([D]\cup j_\ast(\alpha))=b\cdot \gamma([D])\cup j^\sigma_\ast (\phi(\alpha))
\]
for all $\alpha \in H^\ast(T_{g,4},\R)$, where $b$ is, as above, the nontrivial constant with $\gamma([D])=b\cdot [D^\sigma]$.
Applying the same argument to $\gamma^{-1}$, we obtain an $\R$-linear inverse of $\phi$.

By Theorem \ref{thm0'}, $\phi$ cannot be an isomorphism of algebras and so we will obtain a contradiction as soon as we have seen that $\phi$ respects the product structures.
For this purpose, let $\alpha$ and $\beta$ denote even-degree cohomology classes on $T_{g,4}$.
Then, by Lemma \ref{lem:blow-up} and \ref{lem:mult}, it suffices to prove
\[
b\cdot \gamma([D])^3\cup j^\sigma_\ast(\phi(\alpha\cup \beta))=b\cdot \gamma([D])^3\cup j^\sigma_\ast(\phi(\alpha)\cup \phi(\beta)) .
\]
Using (\ref{eq1:mult}), the latter is seen as follows:
\begin{align*}
b\cdot \gamma([D])^3\cup j^\sigma_\ast(\phi(\alpha\cup \beta))&=\gamma([D])^2\cup \gamma([D]\cup j_\ast(\alpha\cup \beta)) \\	
																														&=\gamma([D]^2\cup j_\ast(1)\cup j_\ast(\alpha\cup \beta)) \\
																														&=\gamma([D]\cup j_\ast(\alpha)\cup [D]\cup j_\ast(\beta)) \\
																														&=b^2\cdot \gamma([D])^2\cup j^\sigma_\ast(\phi (\alpha))\cup j^\sigma_\ast(\phi(\beta)) \\
																														&=b^2\cdot \gamma([D])^2\cup j^\sigma_\ast(1)\cup j^\sigma_\ast(\phi(\alpha)\cup \phi(\beta))\\
																														&=b\cdot \gamma([D])^3\cup j^\sigma_\ast(\phi(\alpha)\cup \phi(\beta)) .
\end{align*}
This concludes the proof of Theorem \ref{thm2}.
\end{proof}

\section{Examples with non-isotrivial deformations} \label{sec:deformations}
In this section we prove that the examples in Theorem \ref{thm1} may be chosen to have non-isotrivial deformations.
Here, a family $(X_s)_{s\in S}$ of varieties over a connected base $S$ is called non-isotrivial if there are two points $s_0,s_1\in S$ with $X_{s_0}\ncong X_{s_1}$.
The idea of the proof is to vary the blown-up point $p\in Z$ in the construction of Section \ref{sec:thm2}.
In order to state our result, we write $X\sim Y$ if two varieties $X$ and $Y$ are birationally equivalent.

\begin{theorem} \label{thm1'}
Let $Z$ be a smooth complex projective variety of dimension $\geq 10$.
Then there is a non-isotrivial family $(W_p)_{p\in U}$ of smooth complex projective varieties $W_p$ over some smooth affine variety $U$, and an automorphism $\sigma\in \Aut(\C)$ such that for all $p\in U$:
\[
W_p\sim Z \sim W^\sigma_p\ \ \text{and}\ \ H^{2\ast}(W_p,\R)\ncong H^{2\ast}(W_p^\sigma,\R).
\]
\end{theorem}
\begin{proof} 
As in Section \ref{sec:thm2}, we may pick some $\sigma\in \Aut(\C)$ and some $g\geq 1$ such that 
\[
Z\cong Z^\sigma,\ \ \sigma(i)=i,\ \ \sigma(\zeta_{2g+1})\neq \zeta_{2g+1}\ \ \text{and}\ \ b_2(T_{g,4})> b_2(Z)+4 .
\]

Next, let $U\subseteq Z$ be a Zariski open and dense subset with trivial tangent bundle.
Let $\Delta\subseteq U\times Z$ be the graph of the inclusion $U\hookrightarrow Z$ and consider the blow-up $Bl_{\Delta}(U\times Z)$.
The normal bundle of $\Delta$ in $U\times Z$ is trivial, since $U$ has trivial tangent bundle.
Hence, the exceptional divisor of $Bl_{\Delta}(U\times Z)$ is isomorphic to $\Delta\times \CP^{n-1}$.
Since $n\geq 10$, we may fix an embedding of $\Delta\times T_{g,4}$ into this exceptional divisor and consider the blow-up 
\[
Bl_{\Delta\times T_{g,4}}(Bl_{\Delta}(U\times Z)) .
\] 
Projection to the first coordinate then gives a family 
\[
(W_{p})_{p\in U}
\] 
of smooth complex projective varieties, birational to $Z$. 
Then, for all $p\in U$, the conjugate varieties $W_p$ and $W_p^\sigma$ are as in (\ref{def:W}) and (\ref{def:Wsigma}) respectively.
Thus, $W_p\sim Z$ and $W_p^\sigma \sim Z^\sigma$.
By Theorem \ref{thm2} and since $Z\cong Z^\sigma$, we obtain for all $p\in U$: 
\[
W_p\sim Z \sim W^\sigma_p\ \ \text{and}\ \ H^{2\ast}(W_p,\R)\ncong H^{2\ast}(W_p^\sigma,\R).
\]
To conclude Theorem \ref{thm1'}, it therefore remains to prove

\begin{claim} \label{claim:deformations}
After replacing $Z$ by another representative of its birational equivalence class, and for a suitable choice of $U$, the family $(W_{p})_{p\in U}$ is non-isotrivial.
\end{claim}

Let us prove this claim. 
By the arguments of Theorem \ref{thm2}, one sees that any isomorphism $g:W_p\to W_q$ induces an isomorphism $g^\ast$ on cohomology which respects the decomposition (\ref{eq:H^2W}).
This implies that $g$ respects the exceptional divisors and thus induces an isomorphism of $Z$ which takes $p$ to $q$.

The above argument, applied to $p=q$, shows that $W_p$ admits no automorphism which takes points from the exceptional divisors to $Z-\left\{p\right\}$.
In particular, $W_p$ contains a Zariski open subset with trivial tangent bundle and with two points that cannot be interchanged by an automorphism of $W_p$.
Since $W_p$ is birational to $Z$, we may therefore, after possibly replacing $Z$ by another representative of its birational equivalence class, assume that $U$ already contains points $p$ and $q$ which cannot be interchanged by any automorphism of $Z$.
Then, as we have seen, $W_p$ and $W_q$ are not isomorphic.
This finishes the proof of Claim \ref{claim:deformations} and so concludes Theorem \ref{thm1'}.
\end{proof}

\begin{remark} \label{rem:deformations}
In contrast to Theorem \ref{thm1'}, most of the previously known examples of non-homeomorphic pairs of conjugate varieties tend to be rather rigid and do in general not occur in non-isotrivial families.
This was already observed by D.\ Reed in \cite{reed}.
However, it is often possible to obtain non-isotrivial families as products of previously known examples with non-rigid varieties, e.g.\ one could take products of Serre's examples \cite{serre} with a smooth hypersurface of degree at least $4$ in $\CP^3$, since the latter are simply connected and come in non-isotrivial families.
\end{remark}

\section*{Acknowledgement}
I am very grateful to my advisor Daniel Huybrechts for many useful discussions, for his continued support, and in particular for pointing me to \cite{charles}. 
I would also like to thank the referee for carefully reading my paper.
This paper was carried out while the author was supported by an IMPRS Scholarship of the Max Planck Society.


\begin{thebibliography}{9}   

\bibitem{abelson}
H. Abelson, {\em Topologically distinct conjugate varieties with finite fundamental group}, Topology \textbf{13}, (1974), 161-176.   

\bibitem{catanese}
I. Bauer, F. Catanese and F. Grunewald, {\em Faithful actions of the absolute Galois group on connected components of moduli spaces}, Preprint, arXiv:1303.2248, (2013).

\bibitem{charles}
F. Charles, {\em Conjugate varieties with distinct real cohomology algebras}, J. reine angew. Math. \textbf{630}, (2009), 125-139.

\bibitem{charles-schnell}
F. Charles and C. Schnell, {\em Notes on absolute Hodge classes}, Preprint, arXiv:1101.3647, (2011).

\bibitem{deligne}
P. Deligne, {\em Hodge cycles on abelian varieties} (notes by J. S. Milne), in Lecture Notes in Mathematics \textbf{900}, Springer Verlag, 1982, 9-100.

\bibitem{freitag}
E. Freitag, {\em Siegelsche Modulfunktionen}, Grundlehren der mathematischen Wissenschaften 254, Springer-Verlag, Berlin, 1983.

\bibitem{geemen}
B. v. Geemen, {\em Some equations for the universal Kummer variety}, Preprint, arXiv:1307.2463, (2013).

\bibitem{vakil}
R. W. Easton and R. Vakil,  {\em Absolute Galois acts faithfully on the components of the
moduli space of surfaces: a Belyi-type theorem in higher dimension}, 
Int. Math. Res. Not. IMRN (2007), no. 20, Art. ID rnm080, 10pp.

\bibitem{kollar}
J. Koll\'ar, {\em Shafarevich Maps and Automorphic Forms}, Princeton University Press, Princeton, 1995.

\bibitem{rajan}
C. S. Rajan, {\em An example of non-homeomorphic conjugate varieties}, Math. Res. Lett. \textbf{18}, (2011), 937-943.

\bibitem{reed}
D. Reed, {\em The topology of conjugate varieties}, Math. Ann. \textbf{305}, (1996), 287-309.

\bibitem{ren}
Q. Ren, S. V. Sam, G. Schrader and B. Sturmfels, {\em The universal Kummer threefold}, Experimental Mathematics \textbf{22} (2013), 327-362. 

\bibitem{serre}
J-P. Serre, {\em Exemples de variétés projectives conjuguées non homéomorphes}, C. R. Acad. Sci. Paris \textbf{258} (1964), 4194-4196.

\bibitem{GAGA}
J-P. Serre, 
{\em G\'eom\'etrie alg\'ebrique et g\'eom\'etrie analytique}, Ann. Inst. Fourier, Grenoble \textbf{6}, (1955-1956), 1-42. 

\bibitem{shimada}
I. Shimada {\em Non-homeomorphic conjugate complex varieties}, in ``Singularities-Niigata-Toyama 2007'', \textbf{56}, Adv. Stud. Pure Math., Math. Soc. Japa, Tokyo, (2009), 285-301. 

\bibitem{tashiro}
Y. Tashiro, S. Yamazaki, M. Ito and T. Higuchi, {\em On Riemann's period matrix of $y^2=x^{2n+1}-1$}, RIMS Kokyuroku \textbf{963} (1996), 124-141 (English).

\bibitem{voisin1}
C. Voisin, {\em Hodge Theory and Complex Algebraic Geometry, I}, Cambridge University Press, Cambridge, 2002. 

\bibitem{voisin:inv}
C. Voisin, {\em On the homotopy types of Kähler compact and complex projective manifolds}, Invent. Math. \textbf{157} (2004), no. 2, 329-343.

\bibitem{voisin:aspects}
C. Voisin, {\em Some aspects of the Hodge conjecture}, Jpn. J. Math. \textbf{2} (2007), no. 2, 261-296.

\end{thebibliography}
\end{document}